\documentclass{amsart}
\usepackage{graphicx, amssymb}
\usepackage[alphabetic,y2k,lite]{amsrefs}

\usepackage{tikz}
\usetikzlibrary{shapes.geometric}
\usetikzlibrary{calc}
\usetikzlibrary{scopes}
\usetikzlibrary{decorations.markings}

\tikzset{
every picture/.style={line width=0.8pt, >=stealth,
                       baseline=-3pt,label distance=-3pt},
dotnode/.style={fill=black,circle,minimum size=2.5pt, inner sep=1pt, outer
sep=0},
morphism/.style={circle,draw,thin, inner sep=1pt, minimum size=15pt,
                 scale=0.8},
small_morphism/.style={circle,draw,thin,inner sep=1pt,
                       minimum size=10pt, scale=0.8},
coupon/.style={draw,thin, inner sep=1pt, minimum size=18pt,scale=0.8},
regular/.style={densely dashed},
edge/.style={thick, dashed, draw=blue, text=black},
boundary/.style={thick,  draw=blue, text=black},
overline/.style={preaction={draw,line width=2mm,white,-}},
drinfeld center/.style={>=stealth,green!60!black, double
distance=1pt,text=black},
cell/.style={fill=black!10},
subgraph/.style={fill=black!30},
midarrow/.style={postaction={decorate},
                 decoration={
                    markings,
                    mark=at position #1 with {\arrow{>}},
                 }},
midarrow/.default=0.5
}

\newtheorem*{theorem*}{Theorem}
\newtheorem{theorem}{Theorem}[section]
\newtheorem{lemma}[theorem]{Lemma}

\newtheorem{corollary}[theorem]{Corollary}

\theoremstyle{definition}
\newtheorem{definition}[theorem]{Definition}

\newtheorem{example}[theorem]{Example}

\theoremstyle{remark}
\newtheorem{remark}[theorem]{Remark}

\numberwithin{equation}{section}
\newcommand{\firef}[1]{Figure~{\rm\ref{#1}}}
\newcommand{\thref}[1]{Theorem~{\rm\ref{#1}}}

\newcommand{\leref}[1]{Lemma~{\rm\ref{#1}}}
\newcommand{\coref}[1]{Corollary~{\rm\ref{#1}}}

\newcommand{\deref}[1]{Definition~{\rm\ref{#1}}}
\newcommand{\exref}[1]{Example~{\rm\ref{#1}}}

\newcommand{\seref}[1]{Section~{\rm\ref{#1}}}

\newcommand{\st}{\; | \;}                               
\newcommand{\cc}[1]{\underset{\scriptstyle #1}{\circ}}

\newcommand{\ov}{\overline}
\newcommand{\del}{\partial}
\newcommand{\<}{\langle}
\renewcommand{\>}{\rangle}
\newcommand{\xxto}{\xrightarrow}              

\newcommand{\one}{\mathbf{1}}
\renewcommand{\i}{{\mathrm{i}}}   
\newcommand{\kk}{\mathbf{k}}       
\newcommand{\Z}{\mathbb{Z}}       
\newcommand{\R}{\mathbb{R}}       
\newcommand{\DD}{\mathcal{D}}      
\newcommand{\C}{\mathcal{C}}      
\newcommand{\Chat}{\widehat{\mathcal{C}}}      
\newcommand{\A}{\mathcal{A}}      
\newcommand{\Vect}{\mathcal{V}ec}  
\newcommand{\ee}{\mathbf{e}}       
\newcommand{\VV}{\mathbf{V}}       

\newcommand{\al}{\alpha}
\newcommand{\be}{\beta}

\newcommand{\Ga}{\Gamma}
\newcommand{\de}{\delta}
\newcommand{\De}{\Delta}

\newcommand{\ph}{\varphi}
\newcommand{\Ph}{\Phi} 

\newcommand{\Si}{\Sigma}

\newcommand{\Sihat}{\widehat{\Sigma}}

\newcommand{\Hs}{H^{string}}
\newcommand{\HsD}{H^{string}_\De}
\newcommand{\Hhs}{\hat{H}^{string}}
\newcommand{\HhsD}{\hat{H}^{string}_\De}
\newcommand{\HTV}{H_{TV}}
\newcommand{\ZTV}{Z_{TV}}

\DeclareMathOperator{\Irr}{Irr}
\DeclareMathOperator{\id}{id}

\DeclareMathOperator{\Hom}{Hom}

\DeclareMathOperator{\End}{End}

\DeclareMathOperator{\im}{Im}

\DeclareMathOperator{\ev}{ev} 
\DeclareMathOperator{\coev}{coev} 
\DeclareMathOperator{\Obj}{Obj}
\DeclareMathOperator{\Gr}{Graph}
\DeclareMathOperator{\VGr}{VGraph}

\begin{document}

\title{String-net model of Turaev-Viro invariants}

\author{Alexander Kirillov, Jr.}
   \address{Department of Mathematics, SUNY at Stony Brook, 
            Stony Brook, NY 11794, USA}

    \email{kirillov@math.sunysb.edu}
    \urladdr{http://www.math.sunysb.edu/\textasciitilde kirillov/}
\thanks{This  work was partially suported by NSF grant DMS-0700589 }

\begin{abstract}
In this paper, we describe the relation between the Turaev--Viro TQFT and the
string-net space introduced in the papers of Levin and Wen. In particular,
the case of surfaces with boundary is considered in detail.    
\end{abstract}

\maketitle
\section*{Introduction}
It is known that any spherical fusion  category $\A$ (i.e. a semisimple abelian 
category with finitely many simple objects, an associative tensor product 
and a duality functor satisfying certain properties) defines a 
3-dimensional  topological quantum field theory (TQFT). This construction 
was first  given by Turaev and Viro in \ocite{TV} in the special case of 
the category of representations of the quantum group $U_q sl(2)$ and 
generalized to arbitrary spherical categories in \ocite{barrett}. It was 
recently shown in \ocite{balsam-kirillov}  that this theory can be extended 
to a 3-2-1 theory, i.e. allowing for surfaces with boundary and 3-
cobordisms with ``tubes'', or ``Wilson lines'', connecting the boundary 
circles of the 2-dimensional surfaces. 

In this paper we show that the vector space $Z_{TV}(\Si)$ which this theory 
associates  to a 2-dimensional surface can also be described in terms of 
so-called ``string  nets'', or space of colored graphs on the surface 
modulo some local relations. These  string nets were introduced in the 
paper of Levin and Wen \ocite{levin-wen}; the  corresponding 
model is called the Levin-Wen model. [In fact, this is one of the two 
equivalent descriptions of the Levin--Wen model; in the other description, 
the vector space associated to the surface is described as the ground space 
of certain Hamiltonian. This second description will not be used in this 
paper.] The same model has also  appeared in the works of Kitaev on 
topological quantum computation \ocite{kitaev}; for 
example, in the special case when $\A$ is the category of $\Z_2$-graded 
vector spaces, the model is known as Kitaev's toric code model.  

The idea that Turaev--Viro  and Kitaev--Levin--Wen models are equivalent is 
certainly not new. This statement has been made in a number of papers, most 
notably in \ocite{kadar} and \ocite{kuperberg}. However, none of these 
papers contain full proofs.  In these papers, many of  the  
statements are written in the special case when all multiplicities in 
tensor product of two simple objects are zero or one (with the note that it 
can be generalized)  and some details of the proofs are missing. The goal of 
this paper is to give a complete  and readable to mathematicians proof of 
the above statement.

In addition, we also carefully treat the case of surfaces with boundary, 
which in the  language of Levin-Wen model correspond to ``excited states'', 
or ``quasiparticles''.  We show that these excited states are again 
equivalent to the Turaev--Viro model   for  surfaces with boundary 
as defined in \ocite{balsam-kirillov}; in particular, we  show that  
the possible boundary conditions for a circle are described by 
objects in the category  $\C=Z(\A)$---the Drinfeld center of $\A$. 

\subsection*{Acknowledgments}
The author would like to thank Oleg Viro, Ben Balsam, Zhenghan Wang, Kevin
Walker, Anton Kapustin, Alexey Kitaev and Joel Kamnitzer for numerous
helpful discussions. Special thanks to Microsoft Station Q, where part of
this work was written.  

\subsection*{Notation}
Throughout the paper, we fix an algebraically closed field $\kk$ of
characteristic zero. All vector spaces and linear maps will be  
considered over $\kk$.

All manifolds, homeomorphisms etc. are considered in piecewise-linear (PL)
topology. [Note that it is well known that in dimensions 2 and below, PL
category is equivalent to the smooth category; however, PL setting is more
convenient for our purposes.] We denote by $D^2$ the ``standard''
two-dimensional disk:  
$$
D^2=\{(x,y)\in \R^2\st |x|\le 1,\ |y|\le 1\}
$$
and by $S^1$ its boundary. We will also use the notation $I=[0,1]$. 

Unless otherwise noted, all manifolds are oriented, and homeomorphisms are
orientation-preserving. $D^2$ is considered with the natural orientation
inherited from $\R^2$, and $S^1$ with the counterclockwise orientation. 

In the figures, we use different line styles for different kinds of lines. Note that some lines are colored, so if you are reading this paper printed in black and white, it will be difficult to tell these styles apart:

  
\begin{tikzpicture}\draw[boundary] (0,0)--(1,0); \end{tikzpicture}
  \qquad boundary of the surface
  
\begin{tikzpicture}\draw[edge] (0,0)--(1,0); \end{tikzpicture}
  \qquad edge of a cell decomposition for the surface

\begin{tikzpicture}\draw (0,0)--(1,0); \end{tikzpicture}
  \qquad a graph on the surface

\begin{tikzpicture}\draw[regular] (0,0)--(1,0); \end{tikzpicture}
  \qquad arc of a graph colored in a special way 
   (see Eqn.~\eqref{e:regular_color} )

\begin{tikzpicture}\draw[drinfeld center] (0,0)--(1,0); \end{tikzpicture}
  \qquad arc of a graph colored by an object of Drinfeld center of $\A$ 
  (see   \firef{f:crossing})

\section{Spherical categories: an overview}\label{s:LW1}

In this section we collect notation and some facts about spherical
categories.

We denote by $\Vect$ the category of finite-dimensional vector spaces over
the ground field $\kk$.

Throughout the paper, $\A$ will denote a spherical fusion category over 
 $\kk$. We refer the reader to the paper \ocite{drinfeld} for the
definitions and properties of such categories.

In particular, $\A$ is semisimple with finitely many
isomorphism classes of simple objects. We will denote by $\Irr(\A)$ the
set of isomorphism classes of simple objects; abusing the language, we will
frequently use the same letter $i$ for denoting both a  simple object and 
its isomorphism class. In those cases where it can lead to confusion, we
will use notation $X_i$ for a representative of isomorphism class $i$. We
will also denote by $\one=X_0$ the unit object in $\A$ (which is simple).

To simplify the notation, we will assume that $\A$ is a strict pivotal
category, i.e. that $V^{**}=V$. As is well-known, this is not really a
restriction, since any pivotal category is equivalent to a strict pivotal
category.

We will denote, for an object $X$ of $\A$, by
$$
d_X=\dim X\in \kk
$$
its categorical dimension; it is known that for simple $X$, $d_X$ is
non-zero. We will fix, for any simple object $X\in \A$, a choice of
square root $\sqrt{d_X}$ so that for $X=\one$, $\sqrt{d_\one}=1$ and that
for any simple $X$, $\sqrt{d_X}=\sqrt{d_{X^*}}$. 

We will also denote
\begin{equation}\label{e:DD}
\DD=\sqrt{\sum_{i\in \Irr(\A)}d_i^2}
\end{equation}
(throughout the paper, we fix a choice of the square root). Note that by
results of \ocite{ENO2005}, $\DD\ne 0$.

We define the functor $\A^{\boxtimes n}\to \Vect$ by
\begin{equation}\label{e:vev}
\<V_1,\dots,V_n\>=\Hom_\A(\one,
V_1\otimes\dots\otimes V_n)
\end{equation}
for any collection $V_1,\dots, V_n$ of objects of $\A$. Note that pivotal
structure gives functorial isomorphisms
\begin{equation}\label{e:cyclic}
z\colon\<V_1,\dots,V_n\>\simeq \<V_n, V_1,\dots,V_{n-1}\>
\end{equation}
such that $z^n=\id$ (see \ocite{BK}*{Section 5.3}); thus, up to a canonical
isomorphism, the space $\<V_1,\dots,V_n\>$ only depends on the cyclic order
of $V_1,\dots, V_n$.

We have a natural composition map 
\begin{equation}\label{e:composition}
\begin{aligned}
 \<V_1,\dots,V_n, X\>\otimes\<X^*, W_1,\dots,
W_m\>&\to\<V_1,\dots,V_n, W_1,\dots, W_m\>\\
\ph\otimes\psi\mapsto \ph\cc{X}\psi= \ev_X\circ (\ph\otimes\psi)
\end{aligned}
\end{equation}
where $\ev_X\colon X\otimes  X^*\to \one$ is the evaluation morphism (see
\firef{f:local_rels1} for an illustration of this operation).  

Note that for any objects $A,B\in \Obj \C$, we have a
non-degenerate pairing $\Hom_\C(A,B)\otimes \Hom_\C(A^*,B^*)\to \kk$
defined by
\begin{equation}\label{e:pairing}
(\ph, \ph')=(\one\xxto{\coev_A}A\otimes A^*\xxto{\ph\otimes \ph'}
  B\otimes B^*\xxto{\ev_B}\one)
\end{equation}
In particular, this gives us a non-degenerate pairing
$\<V_1,\dots,V_n\>\otimes \<V_n^*,\dots,V_1^*\>\to \kk$ and thus,
 functorial isomorphisms
\begin{equation}\label{e:dual}
\<V_1,\dots,V_n\>^*\simeq \<V_n^*,\dots,V_1^*\>
\end{equation}
compatible with the cyclic permutations \eqref{e:cyclic}.

\section{Colored graphs}\label{s:colored}
We will consider finite  graphs embedded in an oriented surface $\Si$
(which is not required to be compact and may have boundary); for such a
graph $\Ga$, let $E(\Ga)$ be the set of edges. Note that edges are not
oriented. Let $E^{or}$ be the set of oriented edges, i.e. pairs $\ee=(e,
\text{orientation of } e)$; for such an oriented edge $\ee$, we denote by
$\bar{\ee}$ the edge with opposite orientation.

If $\Si$ has a boundary, the graph is allowed to have uncolored one-valent
vertices on $\del \Si$ but no other common points with $\del \Si$; all
other  vertices will  be called interior.  We will  call the edges of $\Ga$
terminating at these  one-valent vertices {\em legs}.   
\begin{definition}\label{d:coloring} Let $\Si$ an oriented surface
(possibly with boundary) and $\Ga\subset \Si$ --- an embedded graph as
defined above.  A {\em coloring} of $\Ga$ is the
following data:

  \begin{itemize}
    \item Choice of an object $V(\ee)\in \Obj \A$ for every oriented edge
        $\ee\in E^{or}(\Ga)$ so that $V(\ov{\ee})=V(\ee)^*$.
    \item Choice of a vector $\ph(v)\in \<V(\ee_1),\dots,V(\ee_n)\>$ 
      (see \eqref{e:vev})  for    every interior vertex $v$, where 
      $\ee_1, \dots, \ee_n$ are edges incident to $v$, taken in counterclockwise 
      order and with outward orientation (see \firef{f:coloring}). 
\end{itemize}

An {\em isomorphism} $f$ of two coloring $\{V(\ee), \ph(v)\}$, $\{V'(\ee), 
\ph'(v)\}$ is a collection of isomorphisms $f_\ee\colon V(\ee)\simeq 
V'(\ee)$ which  agree  with isomorphisms $V(\ov{\ee})=V(\ee)^*$ and which 
identify $\ph', \ph$:  $\ph'(v)=f\circ\ph(v)$. 

We will denote the set of all colored graphs on a surface $\Si$ by
$\Gr(\Si)$.
\end{definition}

Note that if $\Si$ has a boundary, then every colored graph $\Ga$ defines
a collection of points $B=\{b_1,\dots, b_n\}\subset \del \Si$ (the
endpoints of the legs of $\Ga$) and a collection of objects $V_b\in \Obj\
\A$ for every $b \in B$: the colors of the legs of $\Ga$ taken with
outgoing orientation. We will denote the pair $(B, \{V_b\})$ by
$\VV=\Ga\cap \del\Si$ and call it {\em boundary value}. We will denote  
$$
\Gr(\Si, \VV)=\text{set of all colored graphs in $\Si$ with boundary value
} \VV.
$$ 
We will return to the discussion of possible boundary values  later in
\seref{s:boundary}. 

We can also consider formal linear combinations of colored graphs. Namely,
for fixed boundary value $\VV$ as above, we will denote 
\begin{equation}\label{e:vgr}
\VGr(\Si,\VV)=\{\text{formal linear combinations of graphs }\Ga\in
\Gr(\Si,\VV)\}
\end{equation}
In particular, if $\del \Si=\varnothing$, then the only possible boundary
condition is trivial ($B=\varnothing$); in this case, we wil just write
$\VGr(\Si)$.

In the figures, we will show the coloring by choosing for each edge an 
orientation and writing the  color of the corresponding oriented edge next 
to it; we will also frequently  replace vertices by round circles labeled
by the corresponding  vector $\ph(v)$, as shown in \firef{f:coloring}.

\begin{figure}[ht]
\begin{tikzpicture}
\node[morphism] (ph) at (0,0) {$\ph$};
\draw[->] (ph)-- +(240:1cm) node[pos=0.7, left] {$V_n$} ;
\draw[->] (ph)-- +(180:1cm);
\draw[->] (ph)-- +(120:1cm);
\draw[->] (ph)-- +(60:1cm);
\draw[->] (ph)-- +(0:1cm);
\draw[->] (ph)-- +(-60:1cm) node[pos=0.7, right] {$V_1$};
\end{tikzpicture}
\caption{Labeling of colored graphs}\label{f:coloring}
\end{figure}
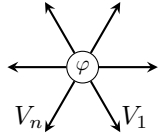

Note that since we have a canonical isomorphism $ \<V, W^*\>\simeq
\Hom_\A(W,V)$, we can also interpret an element $\ph\in \<V_1,\dots, V_n,
W_k^*, \dots, W_1^*\>$ as a morphism $W_1\otimes\dots \otimes W_k \to
V_1\otimes\dots\otimes V_n$, as illustrated in \firef{f:tangle}.
\begin{figure}[ht]
\begin{tikzpicture}
\node[morphism] (ph) at (0,0) {$\ph$};
\draw[->] (ph)-- +(240:1cm) node[pos=0.7, left] {$V_1$} ;
\draw[<-] (ph)-- +(120:1cm) node[pos=0.7, left] {$W_1$};
\draw[<-] (ph)-- +(60:1cm) node[pos=0.7, right] {$W_k$};
\draw[->] (ph)-- +(-60:1cm) node[pos=0.7, right] {$V_n$};
\end{tikzpicture}
$$ \ph\in \<V_1,\dots, V_n, W_k^*,\dots, W_1^*\>\simeq 
\Hom_\A(W_1\otimes \dots\otimes W_k, V_1\otimes\dots\otimes V_n)
$$
\caption{Colored graph as a tangle}\label{f:tangle}
\end{figure}
\begin{remark}
Note that this convention is {\bf opposite} to the conversion used in
\ocites{balsam-kirillov,BK}, where morphisms were acting ``from the bottom
to  top''. Thus, care must be taken when using  graphical presentations of
morphisms  from those papers.  
\end{remark}

We will also use the following convention: if a figure contains a pair of
vertices, one with outgoing edges labeled $V_1,\dots, V_n$ and
the other with edges labeled $V_n^*,\dots, V_1^*$ , and the vertices  are
labeled by the same letter $\al$  (or $\be$, or \dots)
it will stand for summation over the dual bases:
\begin{equation}\label{e:summation_convention}
\begin{tikzpicture}
\node[morphism] (ph) at (0,0) {$\al$};
\draw[->] (ph)-- +(240:1cm) node[pos=0.7, left] {$V^{*}_1$} ;
\draw[->] (ph)-- +(180:1cm);
\draw[->] (ph)-- +(120:1cm);
\draw[->] (ph)-- +(60:1cm);
\draw[->] (ph)-- +(0:1cm);
\draw[->] (ph)-- +(-60:1cm) node[pos=0.7, right] {$V^{*}_n$};
\node[morphism] (ph') at (3,0) {$\al$};
\draw[->] (ph')-- +(240:1cm) node[pos=0.7, left] {$V_n$} ;
\draw[->] (ph')-- +(180:1cm);
\draw[->] (ph')-- +(120:1cm);
\draw[->] (ph')-- +(60:1cm);
\draw[->] (ph')-- +(0:1cm);
\draw[->] (ph')-- +(-60:1cm) node[pos=0.7, right] {$V_1$};
\end{tikzpicture}
\quad = \sum_\al\quad
\begin{tikzpicture}
\node[morphism] (ph) at (0,0) {$\ph_\al$};
\draw[->] (ph)-- +(240:1cm) node[pos=0.7, left] {$V^{*}_1$} ;
\draw[->] (ph)-- +(180:1cm);
\draw[->] (ph)-- +(120:1cm);
\draw[->] (ph)-- +(60:1cm);
\draw[->] (ph)-- +(0:1cm);
\draw[->] (ph)-- +(-60:1cm) node[pos=0.7, right] {$V^{*}_n$};
\node[morphism] (ph') at (3,0) {$\ph^\al$};
\draw[->] (ph')-- +(240:1cm) node[pos=0.7, left] {$V_n$} ;
\draw[->] (ph')-- +(180:1cm);
\draw[->] (ph')-- +(120:1cm);
\draw[->] (ph')-- +(60:1cm);
\draw[->] (ph')-- +(0:1cm);
\draw[->] (ph')-- +(-60:1cm) node[pos=0.7, right] {$V_1$};
\end{tikzpicture}
\end{equation}
where $\ph_\al\in \<V_1,\dots, V_n\>$, $\ph^\al\in \<V_n^*,\dots, V_1^*\>$
are dual bases with respect to pairing \eqref{e:pairing}.

The following theorem is a variation of result of Reshetikhin and Turaev. 
\begin{theorem}\label{t:RT}
  There is a unique  way to assign to every colored
  planar graph $\Ga$ in a disk $D\subset \R^2$ a vector
  \begin{equation}
    \<\Ga\>_D\in\<V(\ee_1),\dots, V(\ee_n)\>
  \end{equation}
  where $\ee_1,\dots, \ee_n$ are the edges of $\Ga$ meeting the boundary
  of $D$ (legs), taken in counterclockwise order and with outgoing orientation,
  so that that following conditions are satisfied:
  \begin{enumerate}
     \item $\<\Ga\>$ only depends on the isotopy  class of $\Ga$.

    \item If $\Ga$ is a single vertex colored by
          $\ph\in \<V(\ee_1),\dots, V(\ee_n)\>$, then $\<\Ga\>=\ph$.
     
    \item Local relations shown in \firef{f:local_rels1} hold. 
\begin{figure}[ht]
\begin{tikzpicture}
\node[morphism] (ph) at (0,0) {$\ph$};
\node[morphism] (psi) at (1,0) {$\psi$};
\node at (-0.7,0.1) {$\vdots$};
\node at (1.7,0.1) {$\vdots$};
\draw[->] (ph)-- +(220:1cm) node[pos=1.0,below,scale=0.8]
{$V_n$};
\draw[->] (ph)-- +(140:1cm) node[pos=1.0,above,scale=0.8]
{$V_1$};
\draw[->] (psi)-- +(40:1cm) node[pos=1.0,above,scale=0.8]
{$W_m$};
\draw[->] (psi)-- +(-40:1cm) node[pos=1.0,below,scale=0.8]
{$W_1$};
\draw[->] (ph) -- (psi) node[pos=0.5,above,scale=0.8] {$X$};
\end{tikzpicture}
=
\begin{tikzpicture}
\node[ellipse, thin, scale=0.8, inner sep=1pt, draw] (ph) at (0,0)
             {$\ph\cc{X}\psi$};
\node at (-0.8,0.1) {$\vdots$};
\node at (0.8,0.1) {$\vdots$};
\draw[->] (ph)-- +(220:1cm) node[pos=1.0,below,scale=0.8] {$V_n$};
\draw[->] (ph)-- +(140:1cm) node[pos=1.0,above,scale=0.8] {$V_1$};
\draw[->] (ph)-- +(40:1cm) node[pos=1.0,above,scale=0.8]  {$W_m$};
\draw[->] (ph)-- +(-40:1cm) node[pos=1.0,below,scale=0.8] {$W_1$};
\end{tikzpicture}
\\
\begin{tikzpicture}
\node[dotnode] (ph) at (0,0) {};
\node[dotnode] (psi) at (1.5,0) {};
\node at (-0.7,0.1) {$\vdots$};
\node at (2.2,0.1) {$\vdots$};
\draw[->] (ph)-- +(220:1cm) node[pos=1.0,below,scale=0.8] {$A_n$};
\draw[->] (ph)-- +(140:1cm) node[pos=1.0,above,scale=0.8] {$A_1$};
\draw[->] (psi)-- +(40:1cm) node[pos=1.0,above,scale=0.8] {$B_m$};
\draw[->] (psi)-- +(-40:1cm) node[pos=1.0,below,scale=0.8] {$B_1$};
\draw[out=45,in=135, midarrow] (ph) to (psi)
                node[pos=0.5,above, scale=0.8] {$V_k$};
\draw[ out=15,in=165, midarrow] (ph) to (psi);
\draw[ out=-15,in=195, midarrow] (ph) to (psi);
\draw[ out=-45,in=225, midarrow] (ph) to (psi)
                node[pos=0.5,below, scale=0.8] {$V_1$};
\end{tikzpicture}
=
\begin{tikzpicture}
\node[dotnode] (ph) at (0,0) {};
\node[dotnode] (psi) at (1.5,0) {};
\node at (-0.7,0.1) {$\vdots$};
\node at (2.2,0.1) {$\vdots$};
\draw[->] (ph)-- +(220:1cm) node[pos=1.0,below,scale=0.8] {$A_n$};
\draw[->] (ph)-- +(140:1cm) node[pos=1.0,above,scale=0.8] {$A_1$};
\draw[->] (psi)-- +(40:1cm) node[pos=1.0,above,scale=0.8] {$B_m$};
\draw[->] (psi)-- +(-40:1cm) node[pos=1.0,below,scale=0.8] {$B_1$};
\draw[ ->] (ph) to (psi)
            node[pos=0.5,above,scale=0.8] {$V_1\otimes \dots\otimes V_k$};
\end{tikzpicture}
\qquad $k\ge 0$\\
\begin{tikzpicture}
\node[ellipse, scale=0.8, inner sep=1pt, draw,thin] (ph) at (0,0)
{$\coev$};
\draw[->] (ph)-- +(180:1cm) node[pos=1.0,above,scale=0.8] {$V$};
\draw[->] (ph)-- +(0:1cm) node[pos=1.0,above,scale=0.8] {$V^*$};
\end{tikzpicture}
=
\begin{tikzpicture}
\draw[->] (2,0)-- (0,0) node[pos=0.5,above,scale=0.8] {$V$};
\end{tikzpicture}
\caption{Local relations for colored graphs.
         Here $\ph\circ\psi$ is defined by \eqref{e:composition}. 
        }\label{f:local_rels1}
\end{figure}

    Local relations should be understood as follows: for any pair 
    $\Ga, \Ga'$ of colored graphs which are identical  outside a subdisk 
	$D'\subset D$, and in this disk are homeomorphic to the graphs
    shown in  \firef{f:local_rels1},  we must have $\<\Ga\>=\<\Ga'\>$. 
   \end{enumerate}
    Moreover, so defined $\<\Ga\>$ satisfies the following properties:
    \begin{enumerate} 
    \item $\<\Ga\>$ is linear in color of each vertex $v$ \textup{(}for 
         fixed colors of edges and other vertices\textup{)}.
    \item $\<\Ga\>$ is additive in colors of edges as shown in 
          \firef{f:linearity}.
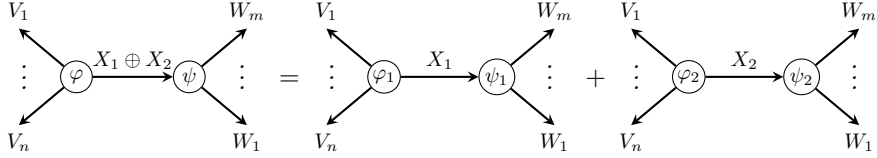
\begin{figure}[ht]
$$
\begin{tikzpicture}
\node[morphism] (ph) at (0,0) {$\ph$};
\node[morphism] (psi) at (1.5,0) {$\psi$};
\node at (-0.7,0.1) {$\vdots$};
\node at (2.2,0.1) {$\vdots$};
\draw[->] (ph)-- +(220:1cm) node[pos=1.0,below,scale=0.8] {$V_n$};
\draw[->] (ph)-- +(140:1cm) node[pos=1.0,above,scale=0.8] {$V_1$};
\draw[->] (psi)-- +(40:1cm) node[pos=1.0,above,scale=0.8] {$W_m$};
\draw[->] (psi)-- +(-40:1cm) node[pos=1.0,below,scale=0.8] {$W_1$};
\draw[->] (ph) -- (psi) node[pos=0.5,above,scale=0.8] {$X_1\oplus X_2$};
\end{tikzpicture}
=
\begin{tikzpicture}
\node[morphism] (ph) at (0,0) {$\ph_1$};
\node[morphism] (psi) at (1.5,0) {$\psi_1$};
\node at (-0.7,0.1) {$\vdots$};
\node at (2.2,0.1) {$\vdots$};
\draw[->] (ph)-- +(220:1cm) node[pos=1.0,below,scale=0.8]{$V_n$};
\draw[->] (ph)-- +(140:1cm) node[pos=1.0,above,scale=0.8]{$V_1$};
\draw[->] (psi)-- +(40:1cm) node[pos=1.0,above,scale=0.8]{$W_m$};
\draw[->] (psi)-- +(-40:1cm) node[pos=1.0,below,scale=0.8]{$W_1$};
\draw[->] (ph) -- (psi) node[pos=0.5,above,scale=0.8] {$X_1$};
\end{tikzpicture}
+
\begin{tikzpicture}
\node[morphism] (ph) at (0,0) {$\ph_2$};
\node[morphism] (psi) at (1.5,0) {$\psi_2$};
\node at (-0.7,0.1) {$\vdots$};
\node at (2.2,0.1) {$\vdots$};
\draw[->] (ph)-- +(220:1cm) node[pos=1.0,below,scale=0.8]{$V_n$};
\draw[->] (ph)-- +(140:1cm) node[pos=1.0,above,scale=0.8]{$V_1$};
\draw[->] (psi)-- +(40:1cm) node[pos=1.0,above,scale=0.8]{$W_m$};
\draw[->] (psi)-- +(-40:1cm) node[pos=1.0,below,scale=0.8]{$W_1$};
\draw[->] (ph) -- (psi) node[pos=0.5,above,scale=0.8] {$X_2$};
\end{tikzpicture}
$$
\caption{Linearity of $\<\Ga\>$. Here $\ph_1,\ph_2$ are compositions
of $\ph$ with projector $X_1\oplus X_2\to X_1$ (respectively, 
$X_1\oplus X_2\to X_2$), and similarly for $\psi_1,\psi_2$.
}\label{f:linearity}
\end{figure}
    
    \item If $\Ga,\Ga'$ are two isomorphic colorings of the same graph,
      then $\<\Ga\>=\<\Ga'\>$. 
    \item Composition property: if $D'\subset D$ is a subdisk such
      that $\del D'$ does not contain vertices of $\Ga$ and meets edges of
      $\Ga$ transversally, then   $\<\Ga\>_D$ will not change if we replace
      subgraph $\Ga\cap D'$ by a single vertex colored by
      $\<\Ga\cap D'\>_{D'}$.

  \end{enumerate}
We will call the vector $\<\Ga\>$ the {\em evaluation} of $\Ga$.
\end{theorem}

In particular, for a planar graph $\Ga\subset \R^2$ with no outgoing
legs, $\<\Ga\>\in \kk$ is a number. 

This evaluation map can be naturally extended to formal linear combinations 
of graphs: for fixed boundary value $\VV=(\{b_1,\dots, b_n\},\{V_1,\dots,
V_n\})$,  the map $\Ga\mapsto \<\Ga\>$ extends in an obvious way to a
linear map 
$$
\VGr(D,\VV)\to \<V_1\otimes\dots\otimes V_n\>
$$

\section{String nets}\label{s:sn}

In this section we give a definition and list some properties of the
main object of our study, the string-net space. Practically all results of
this section are known (with possible  exception of
\leref{l:edge_crossing}); see, e.g., \ocite{kuperberg}).

Let $\Si$ be an oriented   surface; as before, it can have boundary and we 
 do not assume that it is compact --- for example, compact surface with
punctures is also allowed. We fix a boundary value $\VV$ as in the previous
section and consider the set $\Gr(\Si,\VV)$ of  colored graphs in $\Si$
with boundary value $\VV$; we will also use  the vector space $\VGr(
\Si,\VV)$ of formal linear combinations of such graphs.

We now want to define local relations between graphs. One way of doing it 
is as follows.

Let $D\subset \Si$ be an embedded disk, 
$\mathbf\Ga=c_1\Ga_1+\dots+c_n\Ga_n \in \VGr(\Si,\VV)$  --- a linear
combination of colored graphs in $\Si$ such that
\begin{enumerate}
  \item $\Ga$ is transversal to $\del D$ (i.e., no vertices of $\Ga_i$ 
      are on the boundary of $D$ and edges of each $\Ga_i$ meet 
      $\del D$ transversally).
  \item All $\Ga_i$ coincide outside of $D$.
  \item $\<\mathbf{\Ga}\>_D=\sum c_i\<\Ga_i\cap D\>_D=0$, where
      $\<\Ga_i\cap D\>_D$ is the expectation value defined by \thref{t:RT}.
\end{enumerate}
In this case we will call $\mathbf{\Ga}$ a null graph. 

\begin{definition}\label{d:string-net}
 Let $\Si$ be an oriented surface (possibly with boundary) and let 
 $\VV=(B, \{V_b\})$ be a boundary value as defined in  \seref{s:colored}. 
 The string-net space $\Hs(\Si, \VV)$ is the quotient space 
  $$
   \Hs(\Si, \VV)=\VGr(\Si, \VV)/N(\Si, \VV)
  $$
  where $N(\Si, \VV)$ is  the subspace spanned by null graphs 
  (for all possible embedded disks  $D\subset \Si$). 
\end{definition}

\begin{remark}
  This definition is an example  of a general construction of TQFT as 
  space of fields modulo local relations, as  defined in
 \ocite{walker}.
\end{remark}

\begin{example}\label{x:R2}
  Let $\Si=S^2-\{pt\}=\R^2$. Then $\Hs(\Si)=\kk$: the map $\Ga\mapsto
  \<\Ga\>$ descends to  an isomorphism $\Hs\to\kk$.
\end{example}

Motivated by this example, we will denote for a linear combination 
of  colored graphs $\mathbf{\Ga}\in \VGr(\Si, \VV)$ its class in $\Hs(\Si,
\VV)$  by  $\<\Ga\>_\Si$ (or just $\<\Ga\>$ when there is no ambiguity).

It is immediate from the definition that all local  relations listed in
\thref{t:RT} are satisfied in $\Hs$. The following theorem lists some
corollaries of these  relations.

\begin{theorem}\label{t:local_rels2}\par\noindent
\begin{enumerate}
  \item If $\Ga, \Ga'$ are isomorphic coloring of the same graph 
  \textup{(}see \deref{d:coloring}\textup{)}, 
     then $\<\Ga\>=\<\Ga'\>$.
  \item If $\Ga, \Ga'$ are isotopic, then $\<\Ga\>=\<\Ga'\>$.
  \item The map $\Ga\to\<\Ga\>$ is linear in colors of edges and vertices
    in the same sense as in \thref{t:RT}. 
  \item $\Hs(\Si_1\sqcup\Si_2)=\Hs(\Si_1)\otimes \Hs(\Si_2)$
  \item For any surface $\Si$, we have the following local relations 
     in $\Hs(\Si)$: 
    
     \begin{align}
      &\sum_{i\in \Irr(\A)} d_i 
\begin{tikzpicture}
\node[morphism] (ph) at (0,-0.5) {$\al$};
\node[morphism] (psi) at (0,0.5) {$\al$};
\node at (0,1.1) {$\dots$};
\node at (0,-1.1) {$\dots$};
\draw[->] (ph)-- +(-110:1cm) node[pos=1.0,left,scale=0.8]
{$V_1$};
\draw[->] (ph)-- +(-70:1cm) node[pos=1.0,right,scale=0.8]
{$V_n$};
\draw[<-] (psi)-- +(110:1cm) node[pos=1.0,left,scale=0.8]
{$V_1$};
\draw[<-] (psi)-- +(70:1cm) node[pos=1.0,right,scale=0.8]
{$V_n$};
\draw[<-] (ph) -- (psi) node[pos=0.5,left,scale=0.8] {$i$};
\end{tikzpicture}
=
\begin{tikzpicture}
\node at (0,0) {$\dots$};
\draw[<-] (-0.3,-1)-- (-0.3,1) node[pos=0.5,left,scale=0.8]
{$V_1$};
\draw[<-] (0.3,-1)-- (0.3,1) node[pos=0.5,right,scale=0.8]
{$V_n$};
\end{tikzpicture}
          \label{e:local_rels2a}\\
&
\begin{tikzpicture}
\draw[->] (0.5,0) arc (0:180:0.5cm) arc (-180:0:0.5cm);
\node at (45:0.7cm) {$X$};
\end{tikzpicture}
\quad = d_X\\
&
\begin{tikzpicture}
\path[subgraph] (0,0) circle (0.5);
\draw (0.5,0)--(1.5,0) node[pos=0.7, above] {$i$};
\end{tikzpicture}
\quad =0,\qquad i\in \Irr(\A), i\not\simeq \one
     \end{align}
  In the last picture, the shaded area is an embedded disk which  can
contain any subgraph such that the only edge crossing the boundary of the
shaded disk is the one labeled by $i$.  
\end{enumerate}
\end{theorem}
\begin{proof}
 Parts (1)---(3) follow from analogous statements for the disk given in
\thref{t:RT}. (4) is immediate from the definition. Equation
\eqref{e:local_rels2a}  is also well-known; a proof can be found, e.g., in
\ocite{balsam-kirillov}*{Lemma 1.1}. The other two identities immediately
follow from the definition.  
\end{proof}

This theorem has an immediate corollary.
\begin{corollary}\label{c:B_p}
Let dashed line stand for the sum of all colorings of an edge by 
simple objects $i$, each taken with coefficient $d_i$:
     \begin{equation} \label{e:regular_color}
          \begin{tikzpicture}
        \draw[regular] (0, -0.5)--(0, 0.5);
       \end{tikzpicture}
      =\sum_{i\in \Irr(\A)} d_i \quad 
      \begin{tikzpicture}
        \draw (0, -0.5)--(0, 0.5) node[pos=0.8, right] {$i$};
       \end{tikzpicture}
     \end{equation}

Then one has the following relations in $\Hs(\Si)$:
\begin{align}
&\begin{tikzpicture}
  \draw[regular] (0,0) circle (0.5);
 \end{tikzpicture}
 \quad =\DD^2\\
& 
\begin{tikzpicture}
  \node[morphism] (ph) at (0,-0.5) {$\al$};
  \node[morphism] (psi) at (0,0.5) {$\al$};
  \node at (0,1.1) {$\dots$};
  \node at (0,-1.1) {$\dots$};
  \draw[->] (ph)-- +(-110:1cm) node[pos=1.0,left,scale=0.8] {$V_1$};
  \draw[->] (ph)-- +(-70:1cm) node[pos=1.0,right,scale=0.8] {$V_n$};
  \draw[<-] (psi)-- +(110:1cm) node[pos=1.0,left,scale=0.8] {$V_1$};
  \draw[<-] (psi)-- +(70:1cm) node[pos=1.0,right,scale=0.8] {$V_n$};
  \draw[regular] (ph) -- (psi) node[pos=0.5,left,scale=0.8] {}; 
\end{tikzpicture}
  =
  \begin{tikzpicture}
   \node at (0,0) {$\dots$};
   \draw[<-] (-0.3,-1)-- (-0.3,1) node[pos=0.5,left,scale=0.8] {$V_1$};
   \draw[<-] (0.3,-1)-- (0.3,1) node[pos=0.5,right,scale=0.8] {$V_n$};
  \end{tikzpicture}
 \qquad\qquad\\
& 
 \begin{tikzpicture}
   \draw[regular] (0,0) circle(0.4cm);
   \path[subgraph] (0,0) circle(0.3cm); 
   \draw (0,1.2)..controls +(-90:0.8cm) and +(90:0.4cm) ..
                   (-0.6, 0) ..controls +(-90:0.4cm) and +(90:0.8cm) ..
                   (0,-1.2);
 \end{tikzpicture}
\quad=\quad
  \begin{tikzpicture}
    \draw[regular] (0,0) circle(0.4cm);
    \path[subgraph] (0,0) circle(0.3cm); 
    \draw (0,1.2)..controls +(-90:0.8cm) and +(90:0.4cm) ..
                   (0.6, 0) ..controls +(-90:0.4cm) and +(90:0.8cm) ..
                   (0,-1.2);
  \end{tikzpicture}
\end{align}
The last relation holds  regardless  of the contents  of the shaded region
\textup{(}which can contain arbitrary  graphs or punctures\textup{)}. 
\end{corollary}
\begin{proof}
The first two relations are just a rewriting of the relations from
\thref{t:local_rels2}. The final relation   follows by applying the second 
local relation twice as shown below:
$$
 \begin{tikzpicture}
   \draw[regular] (0,0) circle(0.4cm);
   \path[subgraph] (0,0) circle(0.3cm); 
   \draw (0,1.2)..controls +(-90:0.8cm) and +(90:0.4cm) ..
                   (-0.6, 0) ..controls +(-90:0.4cm) and +(90:0.8cm) ..
                   (0,-1.2);
 \end{tikzpicture}
\quad=\quad
\begin{tikzpicture}
  \draw[regular] (0,0) circle(0.4cm);
  \node[dotnode, label=45:$\al$] (top) at (0,0.4) {};
  \node[dotnode, label=-45:$\al$] (bot) at (0,-0.4) {};
  \path[subgraph] circle(0.3cm); 
  \draw (0,1.2)--(top) (bot)--(0,-1.2);    
\end{tikzpicture}
\quad=\quad 
  \begin{tikzpicture}
    \draw[regular] (0,0) circle(0.4cm);
    \path[subgraph] (0,0) circle(0.3cm); 
    \draw (0,1.2)..controls +(-90:0.8cm) and +(90:0.4cm) ..
                   (0.6, 0) ..controls +(-90:0.4cm) and +(90:0.8cm) ..
                   (0,-1.2);
  \end{tikzpicture}
$$
\end{proof}

The following relation will also be useful in the future.
\begin{lemma}\label{l:edge_crossing}
Let $D_1, D_2$ be two non-intersecting disks. Then for any $V,W\in \Obj
\A$, $i\in \Irr(\A)$ we have the following identity
in $\Hs(D_1\sqcup D_2)$:
$$\<
\begin{tikzpicture}
  \node[morphism] (ph) at (0,0) {$\Phi$};
  \node[morphism] (al) at (1,0) {$\al$};
  \draw[->] (-0.5,0)  -- (ph) node[pos=0.5,above, scale=0.8]{$V$};
  \draw[->]  (ph)--(al) node[pos=0.5,above,scale=0.8] {$W$};
  \draw[->] (al)-- (1.5,0) node[pos=0.5,above,scale=0.8] {$i$};
\end{tikzpicture}
\>_{D_1}\otimes 
\<
\begin{tikzpicture}
  \node[morphism] (al) at (0,0) {$\al$};
  \draw[->] (-0.5,0)  -- (al) node[pos=0.5,above, scale=0.8]{$i$};
  \draw[->] (al)-- (0.5,0) node[pos=0.5,above,scale=0.8] {$W$};
\end{tikzpicture}
\>_{D_2}
=
\<
\begin{tikzpicture}
  \node[morphism] (be) at (0,0) {$\beta$};
  \draw(-0.5,0)  -- (be) node[pos=0.5,above, scale=0.8] {$V$};
  \draw[->]  (be)--(0.5,0) node[pos=0.5,above,scale=0.8] {$i$};
\end{tikzpicture}
\>_{D_1}\otimes 
\<
\begin{tikzpicture}
  \node[morphism] (be) at (0,0) {$\be$};
  \node[morphism] (ph) at (1,0) {$\Phi$};
  \draw[->] (-0.5,0)  -- (be) node[pos=0.5,above, scale=0.8] {$i$};
  \draw[->] (be)-- (ph) node[pos=0.5,above, scale=0.8] {$V$};
  \draw[->]  (ph)--(1.5,0) node[pos=0.5,above,scale=0.8] {$W$};
\end{tikzpicture}
\>_{D_2}
$$
\textup{(}as before, we are using summation convention
\eqref{e:summation_convention}\textup{)}. 
\end{lemma}
\begin{proof}
Let us write $V=\bigoplus_j V_j\otimes j$, $W=\bigoplus_j W_j\otimes j$
where $V_j, W_j$ are vector spaces. Clearly, summands with $j\ne i$ give
zero contribution to both sides of equality in the lemma. Thus, it
suffices to prove that for any pair of vector spaces $V_i, W_i$ and a
linear map $\Ph_i\colon V_i\to W_i$, we have
$$
\sum_\al \Ph^*_i(\ph^\al) \otimes \ph_\al
=\sum_\al \psi^\be\otimes \Ph_i(\psi_\be)
\in V_i^*\otimes W_i
$$
where $\ph_\al\in W_i,\ph^\al\in W^*_i$, $\psi_\be\in V_i,\psi^\be\in
V_i^*$ are bases such that $\<\ph_\al,\ph^{\al'}\>
=d_i^{-1}\de_{\al,\al'}$, and similarly for $\psi_\be, \psi^\be$.
This identity is trivial: under identification $V_i^*\otimes W_i\simeq
\Hom(V_i, W_i)$, both sides are identified with  $d_i^{-1}\Ph_i$.
\end{proof}

For future use, we will need to know how the string net space changes 
when we add or remove a puncture. The following lemma, the proof of 
which is left to the reader, is the first step in this direction. 
\begin{lemma}\label{l:puncture}
  Let $\Si'=\Si-p$ be a surface obtained by removing from $\Si$ a single
point $p$. Then the obvious embedding $\Gr(\Si-p)\to
\Gr(\Si)$ descends to an isomorphism
$$
\Hs(\Si)=\Hs(\Si-p)/
\left\langle 
   \begin{tikzpicture}[label distance=-1pt]
   \node[dotnode, label=above:$p$] (v) at (0,0) {};\
   \draw (0,0.6) 
           ..controls +(270:0.3cm) and +(90:0.3cm)..
          (-0.3,0)
            ..controls +(270:0.3cm) and +(90:0.3cm)..
          (0,-0.6);    
   \end{tikzpicture}
\right\rangle 
- \left\langle 
   \begin{tikzpicture}[label distance=-1pt]
   \node[dotnode, label=above:$p$] (v) at (0,0) {};\
   \draw (0,0.6) 
           ..controls +(270:0.3cm) and +(90:0.3cm)..
          (0.3,0)
            ..controls +(270:0.3cm) and +(90:0.3cm)..
          (0,-0.6);    
   \end{tikzpicture}
\right\rangle 
$$
\end{lemma}

\begin{corollary}[\ocite{barrett}]
$\Hs(S^2)=\kk$
\end{corollary}
\begin{proof}
Since $S^2=\R^2\cup\infty$, it suffices to prove that in $\Hs(\R^2)$, 
we have the relation shown in \firef{f:spherical}. But this is part of  the 
definition of a spherical category.

\begin{figure}[ht]
\begin{tikzpicture}
  \draw (0,-1)--(0,1) arc(0:180:0.5cm) -- (-1,-1) node[left] {$X$} 
           arc (-180:0:0.5cm);
  \path[subgraph] circle(0.6cm); 
\end{tikzpicture}
\quad=\quad
\begin{tikzpicture}
  \draw (0,-1)--(0,1) arc(180:0:0.5cm) -- (1,-1) node[left] {$X$} 
           arc (0:-180:0.5cm);
  \path[subgraph] circle(0.6cm); 
\end{tikzpicture}

\caption{Spherical property}\label{f:spherical}
\end{figure}
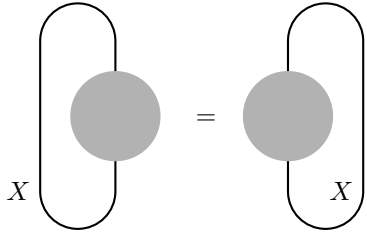

\end{proof}
Of course, this was exactly the motivation for the definition of spherical 
category in \ocite{barrett}. 

Let now $P=\{p_1,\dots, p_k\}$ be a finite collection of distinct points in 
$\Si$. By the \leref{l:puncture}, we have an surjection
$\Hs(\Si-P)\to \Hs(\Si)$. 

\begin{theorem}\label{t:B_p}
Let $P=\{p_1,\dots, p_k\}\subset \Si$. For each point $p\in P$, let 
\begin{equation}\label{e:B_p}
B_p\colon \Hs(\Si-P)\to \Hs(\Si-P)
\end{equation}
be the operator that adds to a colored graph $\Ga$ a small loop around
puncture $p$ colored as  shown in  \firef{f:B_p}. 
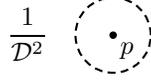
\begin{figure}[ht]
$$
\frac{1}{\DD^2}\quad
\begin{tikzpicture}
\node[dotnode, label=-45:$p$] at (0,0) {};
\draw[regular] (0,0) circle(0.5cm);
\end{tikzpicture}
$$
\caption{Operator $B_p$}\label{f:B_p}
\end{figure}

\begin{enumerate}
\item Each $B_p$ is a projector: $B_p^2=B_p$.
\item Operators $B_{p_i}$ for different points $p_i$ commute. 
      Thus, the operator $B_P=\prod_{p\in P} B_p\colon \Hs(\Si-P)\to
      \Hs(\Si-P)$ is also a projector.
\item The map $\Hs(\Si-P)\to \Hs(\Si)$ gives an isomorphism
$$
  \im(B_P)=\{\psi\in \Hs(\Si-P)\st B_P\psi=\psi\}\simeq \Hs(\Si)
$$
\end{enumerate}
\end{theorem}
\begin{proof}
Part (1) follows from \coref{c:B_p}; part (2) is obvious from the definition. 

To prove part (3), denote by $\pi$ the natural map $\Hs(\Si-P)\to 
\Hs(\Si)$. Then it follows from \coref{c:B_p} and \leref{l:puncture} that 
for any $\psi \in \ker(\pi)$, we have $B_P\psi=0$; thus, $B_P$ is well
defined on $\Hs(\Si)$. Using \coref{c:B_p} again, we see that  $B_P$ acts
by identity on $\Hs(\Si)$:  for every $\psi$, we have
$\pi(B_P\psi)=\pi(\psi)$. This proves the lemma. 
\end{proof}

\section{Turaev-Viro model}\label{s:TV}
In this section we recall the definition of Turaev--Viro model for
an arbitrary spherical fusion category $\A$. The definition, given in
\ocite{barrett}, generalizes the original definition  of Turaev and Viro
given in \ocite{TV}. Our exposition follows our earlier paper
\ocite{balsam-kirillov}, to which the reader is referred for details and
proofs. We give an overivew here for reader's convenience. 

Let $\Si$ be an oriented  closed surface and let $\Delta$ be a cell
decomposition of $\Si$. We will consider not just triangulations but more
general cell decompositions, namely PLCW decompositions as defined in
\ocite{PLCW}. Without going into details, it suffices to say here that
2-cells of such a decomposition  are images of $n$-gons (with $n\ge 1$)
mapped  into $\Si$ so that the map is injective on the interior of the
polygon and also injective on the interior of every edge, but is allowed
to identify different edges or different vertices of the same polygon.
An example of such a cell decomposition would be a 2-torus obtained by
gluing together opposite edges of a rectangle. 

From now on, the words ``cell decomposition'' and ``cell complex'' will
stand for PLCW decomposition and PLCW complex as defined in \ocite{PLCW}.

For every such cell decomposition we can define the state space
$$
\HTV(\Si,\De)=\bigoplus_{l}\bigotimes_C H(C,l)
$$
where $l$ is a coloring of edges of $\De$ by simple objects of $\A$, $C$
is a 2-cell of $\Delta$, and
\begin{equation}
H(C,l)=\<l(e_1), l(e_2),\dots,
l(e_n)\>,\qquad
 \del C=e_1\cup e_2\dots\cup e_n
\end{equation}
where the edges $e_1,\dots, e_n$ are taken in the counterclockwise order on
$\del C$ as shown in \firef{f:state_space1}.
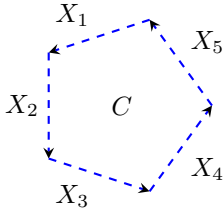
\begin{figure}[ht]
\begin{tikzpicture}
\node at (0,0) {$C$};
\foreach \i in {1,...,5}
   {
    \pgfmathsetmacro{\u}{\i*72}
    \pgfmathsetmacro{\v}{\u+72}
    \draw[edge, ->] (\u:1.2cm)--(\v:1.2cm) node[pos=0.5,auto=right]{$X_\i$};
    }
\end{tikzpicture}
\qquad\qquad $H(C)=\<X_1, \dots,  X_n\>
=\Hom_\A(\one, X_1\otimes\dots\otimes X_n)$
\caption{State space for a cell}\label{f:state_space1}
\end{figure}

Next, given a cobordism $M$ between two surfaces $\Si, \Si'$ with cell
decompositions, one can define an operator $Z(M)\colon \HTV(\Si, \De)\to
\HTV(\Si',\De')$; it is defined using a cell decomposition of $M$ but can
be shown to be independent of the choice of the decomposition (see
\ocite{balsam-kirillov}*{Theorem 4.4}). In
particular, taking $M=\Si\times I$, we get an operator $Z(\Si\times
I)\colon \HTV(\Si,\De)\to\HTV(\Si,\De)$ which can be shown to be a
projector. We now define the Turaev--Viro space associated to $\Si$ as
$$
\ZTV(\Si, \De)=\im(Z(M\times I)).
$$
It can be shown that for any two cell decompositions $\De,\De'$ of the same
surface $\Si$, we have a canonical isomorphism $\ZTV(\Si, \De)\simeq
\ZTV(\Si, \De')$ (see \ocite{balsam-kirillov}*{}); thus, this space is
determined just by the surface $\Si$. Therefore, we will omit $\De$ in the
notation, writing just $\ZTV(\Si)$. 
\section{TV=string nets}\label{s:main}
In this section we will prove the first main result of the paper.
\begin{theorem}\label{t:main}
  Let $\Si$ be a closed oriented surface. Then one has a canonical
  isomorphism
  $$
    \Hs(\Si)\simeq\ZTV(\Si).
  $$
\end{theorem}

The proof of the theorem occupies the rest of this section. Throughout the
proof, we assume that $\Si$ is closed (i.e., it is compact without
boundary). 

We begin by choosing a cell decomposition $\De$ of $\Si$. Then we have a
natural map $\pi_\De\colon \HTV(\Si, \De)\to \Hs(\Si)$
defined as follows. Let $\Ga_\De$ be the dual graph of $\De$. Then each 
coloring $l$ of edges of $\De$ defines a coloring of edges of $\Ga_\De$, and 
for every 2-cell $C$ of $\De$,  a vector $\Ph\in H(C,l)$ defines a coloring 
of the vertex $v$ of $\Ga_\De$ corresponding to $C$, as shown in 
\firef{f:dual_graph}; thus, for a fixed choice of colors $l$ of edges, 
every vector  $\Psi\in \bigotimes_{C}H(C,l)$ defines a 
coloring of $\Ga_\De$ which we will  denote by  $\tilde \Psi$.

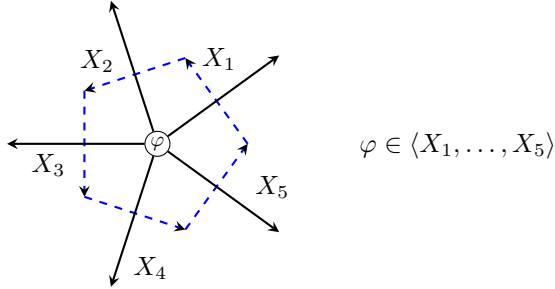
\begin{figure}[ht]
\begin{tikzpicture}
\node[small_morphism] (O) at (0,0) {$\ph$};
\foreach \i in {1,...,5}
   {
    \pgfmathsetmacro{\u}{(\i-1)*72}
    \pgfmathsetmacro{\v}{\u+72}
    \pgfmathsetmacro{\w}{\u+36}
    \draw[edge, ->] (\u:1.2cm)--(\v:1.2cm);
    \draw[->] (O)--(\w:2cm)  node[pos=0.7,auto=left]{$X_\i$};
    }
\end{tikzpicture}
\qquad 
$\ph\in \<X_1,\dots, X_5\>$
\caption{Coloring of the dual graph.}\label{f:dual_graph}
\end{figure}

 Define now the map $\pi_\De$ by 
 \begin{equation}\label{e:TVtoHS}
\begin{aligned}
\pi_\De\colon \HTV(\Si, \De)&\to \Hs(\Si)\\
           \Psi&\mapsto (\Ga_\De, \sqrt{d_l} \tilde \Psi), \qquad 
             \sqrt{d_l}=\prod_{e} d_{l(e)}^{\tfrac{1}{2}}
\end{aligned}
\end{equation}
where the product is over all (unoriented)  edges $e$ of $\De$ and
$d_{l(e)}$ is the dimension of the color $l(e)$ (it does not depend on the
choice of orientation).   

The map \eqref{e:TVtoHS} can be rewritten as follows. Let $\De^0$ 
be the set of vertices of the cell decomposition $\De$  and let 
$$
\HsD=\Hs(\Si-\De^0).
$$
Then we have a natural surjective map $\HsD\to \Hs(\Si)$ (see
\leref{l:puncture}), and the map \eqref{e:TVtoHS} can be written as  the
composition 
\begin{equation}\label{e:TVtoHS2}
\HTV(\Si,\De)\to \HsD\to \Hs(\Si).
\end{equation}

Now \thref{t:main} follows immediately follows from the three lemmas below. 

\begin{lemma} \label{l:main1}
  Let $\De, \De'$ be two different cell decompositions of
  $\Si$, and $f_{\De,\De'}=Z(\Si\times I)\colon H(\Si,\De)\to H(\Si,\De')$ 
  be the canonical linear map defined by the cylinder $\Si\times I$ 
  with any cell decomposition extending $\De,\De'$ on the boundary 
  \textup{(}see\ocite{balsam-kirillov}*{Theorem 4.4}\textup{)}. Then 
  the following  diagram is commutative:
$$
\begin{tikzpicture}
\node (a) at (0, 0) {$H(\Si,\De)$};
\node (b) at (0, -2) {$H(\Si,\De')$};
\node (c) at (3, -1) {$\Hs(\Si)$};
\draw[->, >=latex] (a)--(b) node[pos=0.5,left] {$f_{\De,\De'}$};
\draw[->, >=latex] (a)--(c) node[pos=0.5,above] {$\pi_{\De}$};
\draw[->, >=latex] (b)--(c) node[pos=0.5,above] {$\pi_{\De'}$};
\end{tikzpicture}
$$
\end{lemma}

\begin{lemma}\label{l:main2}
   The map $\HTV(\Si,\De)\to \HsD$ defined by  \eqref{e:TVtoHS2} is an
   isomorphism.
\end{lemma}

\begin{lemma}\label{l:main3}
   The isomorphism  $\HTV(\Si,\De)\to \HsD$ defined by  
   \eqref{e:TVtoHS2} identifies the operator 
   $A=Z(\Si\times I)\colon \HTV(\Si,\De)\to \HTV(\Si,\De)$ with the operator 
   $$
   B_{\De^0}=\prod_{p\in \De^0} B_p\colon  \HsD\to \HsD,
   $$ 
   where 
   $\De^0$ is the set of vertices of $\De$ and operators $B_p$ are defined
   in \thref{t:B_p}.   
   \end{lemma}

Combining these lemmas with the result of \thref{t:B_p}, 
we get the statement of \thref{t:main}. 

We now proceed to the proofs of the three lemmas above.

\begin{proof}[Proof of \leref{l:main1}]
  By the results of \ocite{PLCW}, any two cell decompositions of $\Si$ are
related by a sequence of elementary moves M1 (erasing a vertex of valency
2) and M2 (erasing an edge separating two different cells). Thus, it
suffices to prove the result in these two special cases, where it is an
easy explicit computation.  
\end{proof}

\begin{proof}[Proof of \leref{l:main2}]
We will prove the theorem by constructing an inverse map 
$\HsD\to \HTV(\Si,\De)$. To do that, we first give a different description 
of $\HsD$.

\begin{lemma}\label{l:HsD}
Let $\VGr_\De$ be the space of formal linear combinations of colored graphs 
in $\Si-\De^0$ which are transversal to edges of $\De$ \textup{(}i.e., no vertex 
of $\Ga$ is on an edge of $\De$, and all edges of $\Ga$ intersect edges 
of $\De$ transversally\textup{)}. Then 
$$
\HsD=\VGr_\De/N
$$ 
where $N$ is the subspace generated by 
\begin{itemize}
  \item Local relations inside each 2-cell of $\De$
  \item Local move of moving a vertex through an edge of $\De$ as shown
    in \firef{f:HsD}.
  \begin{figure}[ht]

$$
\begin{tikzpicture}
\node[dotnode] (ph) at (0,0) {};
\draw[->] (ph)-- +(220:1cm);
\draw[->] (ph)-- +(180:1cm);
\draw[->] (ph)-- +(140:1cm);
\draw[->] (ph)-- +(40:1cm);
\draw[->] (ph)-- +(0:1cm);
\draw[->] (ph)-- +(-40:1cm);
\draw[edge] (0.2, -1)--(0.2,1);
\end{tikzpicture}
=
\begin{tikzpicture}
\node[dotnode] (ph) at (0,0) {};
\draw[->] (ph)-- +(220:1cm);
\draw[->] (ph)-- +(180:1cm);
\draw[->] (ph)-- +(140:1cm);
\draw[->] (ph)-- +(40:1cm);
\draw[->] (ph)-- +(0:1cm);
\draw[->] (ph)-- +(-40:1cm);
\draw[edge] (-0.2, -1)--(-0.2,1);
\end{tikzpicture}
$$
    \caption{A defining relation for $\HsD$}\label{f:HsD}
  \end{figure}
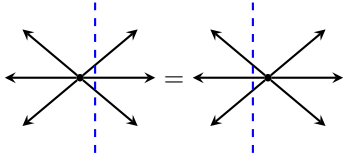
\end{itemize}
\end{lemma}
Proof of this lemma is left to the reader as an easy exercise.

We now define the inverse map $\HsD\to \HTV(\Si,\De)$ as follows. Let $\Ga$
be a  colored graph  satisfying the conditions of \leref{l:HsD}. For any
simple coloring $l$ of edges of $\De$, let $\Ga^l$ be the the formal linear
combination of colored graphs obtained  from $\Ga$ by replacing, for every
edge $e$ of  $\De$, all edges of $\Ga$ crossing $e$ by a single edge
colored by $l(e)$, as shown in \firef{f:edge_tr}.   
\begin{figure}[ht]
$$
\begin{tikzpicture}
\draw[,->] (-1.2,0.3)--(1.2,0.3)  node[pos=1.0,above,scale=0.8] {$V_1$};
\draw[,->] (-1.2,-0.3)--(1.2,-0.3) node[pos=1.0,above,scale=0.8] {$V_n$};
\draw[edge] (0,1)--(0,-1);
\end{tikzpicture}
\qquad \mapsto \qquad
\begin{tikzpicture}
\node[morphism] (ph) at (-0.7,0) {$\al$};
\node[morphism] (psi) at (0.7,0) {$\al$};
\draw[<-] (ph)-- +(150:1cm) node[pos=1.0,left,scale=0.8] {$V_1$};
\draw[<-] (ph)-- +(210:1cm) node[pos=1.0,left,scale=0.8] {$V_n$};
\draw[->] (psi)-- +(30:1cm) node[pos=1.0,right,scale=0.8] {$V_1$};
\draw[->] (psi)-- +(-30:1cm) node[pos=1.0,right,scale=0.8] {$V_n$};
\draw (ph) -- (psi) node[pos=0.6,above]  {$l(e)$};
\draw[edge] (0,1)--(0,-1);
\end{tikzpicture}
$$
\caption{Transformation $\Ga\mapsto\Ga^l$.}\label{f:edge_tr}
\end{figure}
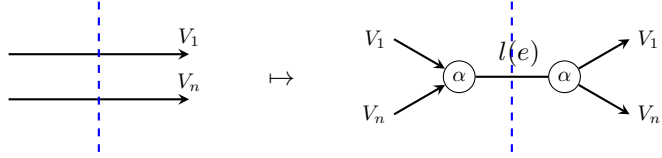

It is immediate from definition that if we consider the intersection 
$\Ga^l\cap C$, where $C$ is a 2-cell of $\De$, then the expectation value  
$$
\<\Ga^l\cap C\>_C\in \<l(e_1), \dots, l(e_n)\>=H(C,l)
$$
where, as in \seref{s:TV}, $l(e_1),\dots, l(e_n)$ are edges of $C$ taken in
counterclockwise order. Thus, we can define the map  
\begin{align*}
s\colon \VGr_\De&\to \HTV(\Si,\De)\\
             \Ga&\mapsto \sum_{l} \sqrt{d_l} \bigotimes_{C} \<\Ga^l\>_C
\end{align*}

It easily follows from \leref{l:HsD} and \leref{l:edge_crossing} that $s$
descends to the string net space $\HsD$. Local relations
\eqref{e:local_rels2a} imply that $s$ is inverse of the map $\pi\colon 
\HTV(\Si,\De)\to\HsD$, which completes the proof of \leref{l:main2}. 
\end{proof}

\begin{proof}[Proof of \leref{l:main3}]
We need to prove that for any $\ph\in \HTV(\Si,\De)$ we have 
$$
\pi_\De(A\ph)=(\prod B_p) \pi_\De(\ph)
$$
where $A=Z_{TV}(\Si\times I)$. 

To avoid complicated notation, we illustrate the proof using the cell
decomposition below.  

Let $\ph=\bigotimes_C \ph_C\in \HTV(\Si,\De)$, so that 
$$
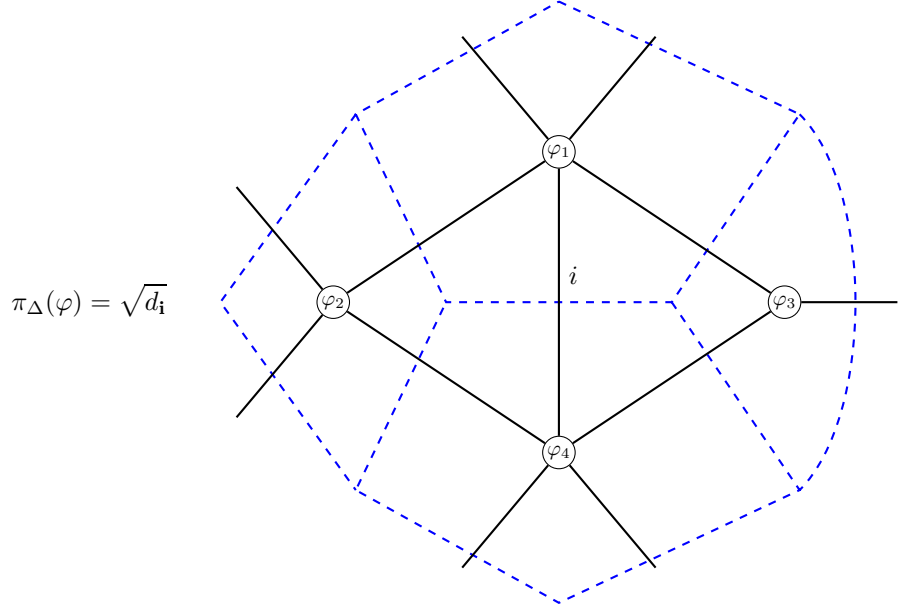
\pi_\De(\ph) = 
\sqrt{d_{\mathbf i}} \qquad
\begin{tikzpicture}
\coordinate (a1) at (1.5,4);
\coordinate (a2) at (-1.2,2.5);
\coordinate (a3) at (4.7,2.5);
\coordinate (a4) at (-3,0);
\coordinate (a5) at (0,0);
\coordinate (a6) at (3,0);
\coordinate (a7) at (-1.2,-2.5);
\coordinate (a8) at (4.7,-2.5);
\coordinate (a9) at (1.5,-4);
\draw[edge] (a1)--(a3)--(a6)--(a5)--(a2)--(a1)
            (a6)--(a8) --(a9)--(a7)--(a5)
            (a2)--(a4)--(a7)
            (a3).. controls +(-45:1.4cm) and +(45:1.4cm)..(a8);
\node[morphism] (ph1) at (1.5,2) {$\ph_1$};
\node[morphism] (ph2) at (-1.5,0) {$\ph_2$};
\node[morphism] (ph3) at (4.5,0) {$\ph_3$};
\node[morphism] (ph4) at (1.5,-2) {$\ph_4$};
\draw (ph1) --(ph2) -- (ph4) -- (ph3) --(ph1) --(ph4) node[pos=0.4, right] {$i$};
\draw (ph1) -- +(130:2cm);
\draw (ph1) -- +(50:2cm);
\draw (ph2) -- +(130:2cm);
\draw (ph2) -- +(230:2cm);
\draw (ph3) -- +(0:1.5cm);
\draw (ph4) -- +(-130:2cm);
\draw (ph4) -- +(-50:2cm);
\end{tikzpicture}
$$
where ${\mathbf i}=(i_1,\dots, i_n)$ is the collection of colors of the edges of the dual graph $\Ga_\De$, and $\sqrt{d_{\mathbf i}}=\prod \sqrt{d_{i_e}}$.
(For illustration we have marked just one of the colors $i_e$ in the
figure). 

Let us consider the cylinder $\Si\times I$ with the obvious cell
decomposition: its cells are of the form $C\times I$, where $C$ is a cell
of $\De$, and let $A=Z(\Si\times I)\colon \HTV(\Si,\De)\to \HTV(\Si,\De)$
be the corresponding operator as defined in \seref{s:TV}.  Then an easy
explicit computation using results of\ocite{balsam-kirillov} shows that  
$$
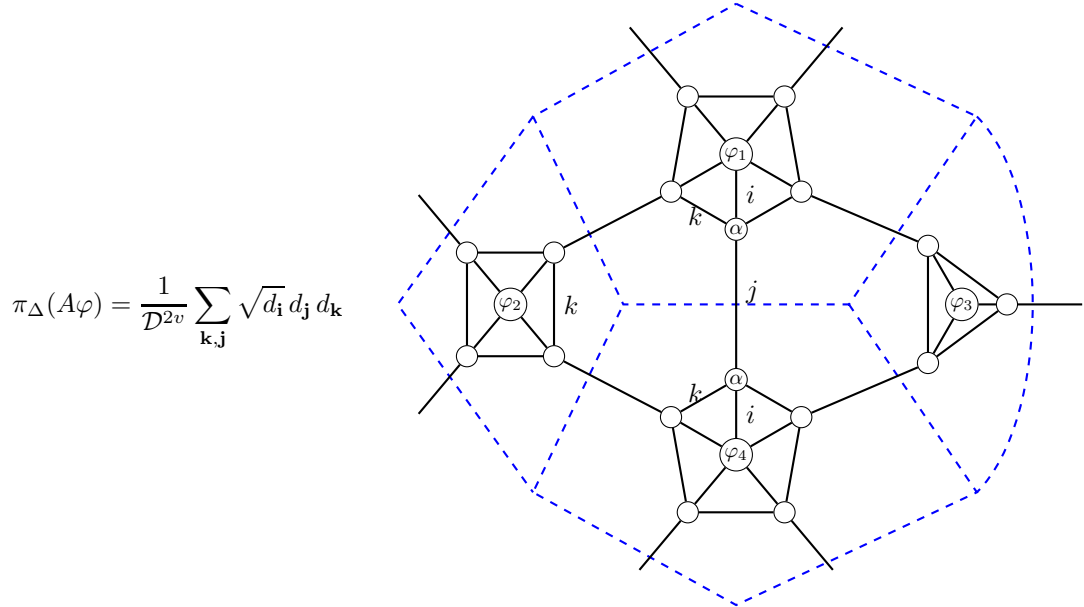
\pi_\De(A\ph) = \frac{1}{\DD^{2v}}
\sum_{\mathbf k,\mathbf  j} \sqrt{d_{\mathbf i}} \, d_{\mathbf j}\, d_{\mathbf k}  \qquad
\begin{tikzpicture}
\coordinate (a1) at (1.5,4);
\coordinate (a2) at (-1.2,2.5);
\coordinate (a3) at (4.7,2.5);
\coordinate (a4) at (-3,0);
\coordinate (a5) at (0,0);
\coordinate (a6) at (3,0);
\coordinate (a7) at (-1.2,-2.5);
\coordinate (a8) at (4.7,-2.5);
\coordinate (a9) at (1.5,-4);
\draw[edge] (a1)--(a3)--(a6)--(a5)--(a2)--(a1)
            (a6)--(a8) --(a9)--(a7)--(a5)
            (a2)--(a4)--(a7)
            (a3).. controls +(-45:1.4cm) and +(45:1.4cm)..(a8);
\node[morphism] (ph1) at (1.5,2) {$\ph_1$};
  \path (ph1)  node[small_morphism] (al1) at +(130:1cm){}
               node[small_morphism] (al2) at +(50:1cm){}
               node[small_morphism] (al3) at +(210:1cm){}
               node[small_morphism] (al4) at +(-30:1cm){}
               node[small_morphism] (al5) at +(270:1cm){$\al$};
 \draw (al1)--(al2)--(al4)--(al5)--(al3) node[pos=0.3,left]{$k$}
       --(al1);
 \draw (ph1) --(al1) (ph1)--(al2) (ph1)--(al3) (ph1)--(al4) 
     (ph1)--(al5) node[pos=0.5,right]{$i$} ;
\node[morphism] (ph2) at (-1.5,0) {$\ph_2$};
  \path (ph2)  node[small_morphism] (al11) at +(130:0.9cm){}
               node[small_morphism] (al12) at +(50:0.9cm){}
               node[small_morphism] (al13) at +(-130:0.9cm){}
               node[small_morphism] (al14) at +(-50:0.9cm){};
 \draw (al11)--(al12)--(al14) node[pos=0.5,right]{$k$}
       --(al13)--(al11);
 \draw (ph2) --(al11) (ph2)--(al12) (ph2)--(al13) (ph2)--(al14);

\node[morphism] (ph3) at (4.5,0) {$\ph_3$};
  \path (ph3)  node[small_morphism] (al15) at +(120:0.9cm){}
               node[small_morphism] (al16) at +(0:0.6cm){}
               node[small_morphism] (al17) at +(-120:0.9cm){};
 \draw (al15)--(al16)--(al17)--(al15);
 \draw (ph3)--(al15) (ph3)--(al16) (ph3)--(al17);
\node[morphism] (ph4) at (1.5,-2) {$\ph_4$};
  \path (ph4)  node[small_morphism] (al6) at +(90:1cm){$\al$}
               node[small_morphism] (al7) at +(150:1cm){}
               node[small_morphism] (al8) at +(30:1cm){}
               node[small_morphism] (al9) at +(230:1cm){}
               node[small_morphism] (al10) at +(-50:1cm){};
 \draw (al6)--(al8)--(al10)--(al9)--(al7)--(al6) node[pos=0.7,left]{$k$};
 \draw (ph4) --(al6) node[pos=0.5,right]{$i$} 
       (ph4)--(al7) (ph4)--(al8) (ph4)--(al9) (ph4)--(al10);
\draw (al5)-- (al6) node[pos=0.4,right]{$j$}
      (al3)--(al12)
      (al14)--(al7)
      (al4)--(al15)
      (al8)--(al17); 
\draw (al1)-- +(130:1.2cm)
      (al2)-- +(50:1.2cm)
      (al11)-- +(130:1cm)
      (al13)-- +(230:1cm)
      (al10)-- +(-50:1cm)
      (al9)-- +(-130:1cm)
      (al16) -- + (0:1cm);
\end{tikzpicture}
$$
where $v$ is the number of vertices of the cell decomposition $\De$ and $\mathbf j, \mathbf k$ are described below. 

In this figure there are 3 kinds of edges: 

``inner'' (they come in pairs of the same color; one such pair is labeled
by $i$ in the figure). We denote by $\mathbf i$ the set of colors of all such pairs. 

``outer'' (one such is labeled by letter $j$). We denote by $\mathbf j$ the set of colors of all such edges

``side'' (these edges are dual to the edges $p\times I$ of $\Si\times I$,
where $p$ is a vertex of $\De$; they come in ``tuples'' --- as many edges
of the same color as there are cells incident to a given vertex. One such
``tuple'' is labeled $k$ in the figure). We denote by $\mathbf k$ the set of colors of all such ``tuples''

The small circles correspond to 2-cells $e\times I$ of $\Si\times I$, where
$e$ is an edge of $\De$; they come in pairs (one such pair is labeled by
letter $\al$ in the figure; as before, we assume summation over dual
bases).

Now, using relation \eqref{e:local_rels2a} we can contract all ``outer''
edges, rewriting this as 
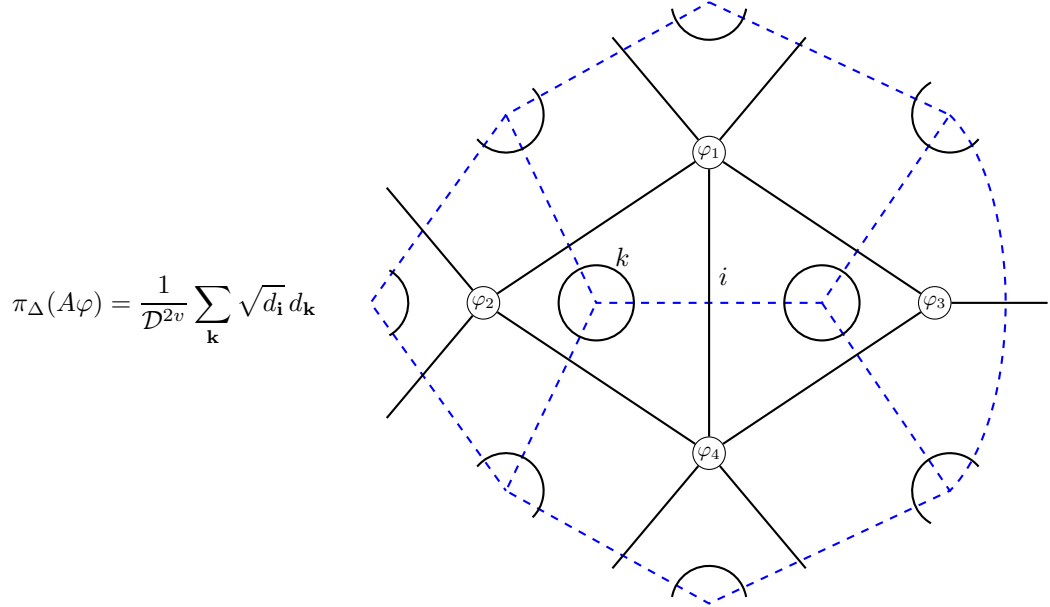
\begin{align*}
\pi_\De(A\ph) &= \frac{1}{\DD^{2v}}
\sum_{\mathbf k}
 \sqrt{d_{\mathbf i}} \,  d_{\mathbf k}  \qquad
\begin{tikzpicture}
\coordinate (a1) at (1.5,4);
\coordinate (a2) at (-1.2,2.5);
\coordinate (a3) at (4.7,2.5);
\coordinate (a4) at (-3,0);
\coordinate (a5) at (0,0);
\coordinate (a6) at (3,0);
\coordinate (a7) at (-1.2,-2.5);
\coordinate (a8) at (4.7,-2.5);
\coordinate (a9) at (1.5,-4);
\draw[edge] (a1)--(a3)--(a6)--(a5)--(a2)--(a1)
            (a6)--(a8) --(a9)--(a7)--(a5)
            (a2)--(a4)--(a7)
            (a3).. controls +(-45:1.4cm) and +(45:1.4cm)..(a8);
\node[morphism] (ph1) at (1.5,2) {$\ph_1$};
\node[morphism] (ph2) at (-1.5,0) {$\ph_2$};
\node[morphism] (ph3) at (4.5,0) {$\ph_3$};
\node[morphism] (ph4) at (1.5,-2) {$\ph_4$};
\draw (ph1) --(ph2) -- (ph4) -- (ph3) --(ph1) --(ph4) node[pos=0.4, right] {$i$};
\draw (ph1) -- +(130:2cm);
\draw (ph1) -- +(50:2cm);
\draw (ph2) -- +(130:2cm);
\draw (ph2) -- +(230:2cm);
\draw (ph3) -- +(0:1.5cm);
\draw (ph4) -- +(-130:2cm);
\draw (ph4) -- +(-50:2cm);
\draw (a5) circle (0.5cm) node at +(60:0.7cm) {$k$}
      (a6) circle(0.5cm) 
      (a1) +(-170:0.5cm) arc (-170:-10:0.5cm)
      (a9) +(170:0.5cm) arc (170:10:0.5cm)
      (a2) +(-140:0.5cm) arc (-140:45:0.5cm)      
      (a7) +(140:0.5cm) arc (140:-45:0.5cm)      
      (a3) +(120:0.5cm) arc (120:320:0.5cm)
      (a8) +(-120:0.5cm) arc (-120:-320:0.5cm)
      (a4) +(60:0.5cm) arc (60:-60:0.5cm);
\end{tikzpicture}
\\
&=(\prod B_p )\pi_\De(\ph)
\end{align*}

\end{proof}

This completes the proof of \thref{t:main}.

\section{Category of boundary values}\label{s:boundary}
We now consider string nets on surfaces with boundary. In this section we 
will show that possible boundary values are described by objects in a 
certain abelian  category $\C(\del \Si)$; choosing a homeomorphism 
$\del\Si\simeq (S^1)^n$ gives an equivalence $\C(\Si)\simeq \C^{\boxtimes
n}$, where $\C=\C(S^1)$ is the so-called Drinfeld center of $\A$.  Thus,
for any choice of objects $Y_1,\dots, Y_n\in \C$ we will have the vector
space $Z(\Si; Y_1,\dots, Y_n)$ of string nets satisfying boundary condition
given by $Y_1\boxtimes\dots \boxtimes Y_n$. 

In physics literature, the space $Z(\Si; Y_1,\dots, Y_n)$  is usually 
thought of as  the space of ``excited states'' on the surface $\hat\Si$ 
obtained by gluing a disk $D_i$ to every boundary 
circle of $\Si$; more precisely, vectors in $Z(\Si; Y_1,\dots, Y_n)$ are 
said to describe states which have excitations localized in disks $D_i$, 
with $Y_i$ describing the type of such an excitation.
Simple objects $Y_i$ are said to describe ``quasiparticle states'' or 
``anyons'' (see \ocite{kitaev}). 

We begin by explaining the reasoning that leads us to the correct 
definition of the category $\C(\del \Si)$; readers who are interested in 
the final answer can skip to \deref{d:Chat}.

Let us for simplicity assume that $S=\del \Si$ is a single circle. Since 
the string nets are allowed to meet the boundary circle, the first natural 
idea is to say that the boundary condition  is described by a collection 
$B=\{b_1,\dots, b_k\}$ of marked points on $S$  colored by objects 
$V_1,\dots, V_k$ of $\A$,a s was done in \seref{s:colored}.  However, it is
also natural to impose the local relation shown in
\firef{f:boundary_move}. 
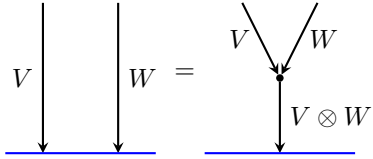
\begin{figure}[ht]
\begin{tikzpicture}[baseline=1cm]
\draw[<-] (-0.5,0)-- (-0.5,2cm) node[pos=0.5,left] {$V$};
\draw[<-] (0.5,0)-- (0.5,2cm) node[pos=0.5,right] {$W$};
\draw[boundary] (-1,0)--(1,0);
\end{tikzpicture}
=
\begin{tikzpicture}[baseline=1cm]
\node[dotnode] (ph) at (0,1) {};
\draw[->] (ph)-- +(0,-1) node[pos=0.5,right] {$V\otimes W$};
\draw[<-] (ph)-- (-0.5,2) node[pos=0.5,left] {$V$};
\draw[<-] (ph)-- (0.5,2) node[pos=0.5,right] {$W$};
\draw[boundary] (-1,0)--(1,0);
\end{tikzpicture}
  \caption{A local move for string nets on a surface with boundary. Here  
    the vertex should be colored by element in $\<W^*,V^*, V\otimes W\>$ 
    corresponding to the identity morphism $V\otimes W\to V\otimes W$.}
  \label{f:boundary_move}
\end{figure}

Thus, the boundary condition described by a collection of objects $V_1, 
\dots, V_i, V_{i+1},\dots$ should be considered equivalent to the boundary 
condition described by the collection  $V_1, \dots, V_{i}\otimes V_{i+1}, 
\dots$.  Since using this equivalence, any collection of 
objects $V_1,\dots, V_k$ can be replaced by a single object $V=V_1\otimes
\dots \otimes V$, one can think that the boundary condition is described 
just by an object $V\in \A$; this is the point of view taken in
\ocite{kuperberg}.

Unfortunately, this is not the right definition. To see why it is so,
consider the case when we have two  points on the boundary circle colored
by objects $V, W\in \A$. Then, as described above,  
this boundary condition is equivalent to the one described by a single
point colored by $V\otimes W$. On the other hand, we could also move point
colored by $W$ around the circle,  arriving at the  pair $W,V$, which is
equivalent to $W\otimes V$. Thus, in our yet to be defined category of
boundary conditions, we must have the relation $V\otimes W\simeq  
W\otimes V$. Informally, we could say that  
this category should be something like the quotient 
\begin{equation}\label{e:quotient}
  \C=\A/(V\otimes W\simeq W\otimes V)
\end{equation}
Of course, defining a ``quotient category'' is more complicated than just 
defining a quotient of a vector space, so one needs to make sense out of 
this formula. One way to do it is to use the 
results of \ocite{greenough}, which gives the answer: when properly 
defined, the quotient \eqref{e:quotient} is exactly the Drinfeld center of 
$\A$. However, we choose to use a different approach (which, of course, 
gives the same answer).

We can now give a precise definition. 
\begin{definition}\label{d:Chat}
  Let $S$ be an oriented 1-dimensional manifold (not necessarily  
  connected). Define  $\Chat(S)$ as the category whose objects are finite 
  subsets  $B\subset S$ together with a choice of object $V_b\in \Obj \A$ for 
  every point $b\in B$; we will use a notation $\VV=(B, \{V_b\})$ for 
  such an object. Define the  morphisms in $\Chat(S)$ by  
  $$
  \Hom_{\Chat}(\VV, \VV')=\Hs(S\times I; \VV^*,\VV'),\qquad
  \VV=(B,\{V_b\}), \quad 
  \VV'=(B',\{V_{b'}\})
 $$ 
 where $\VV^*, \VV'$ means the boundary condition obtained by putting
points $b\in B$ on the ``top'' $S\times\{1\}$, colored by objects
$V_b^*$ for outgoing legs (and thus colored by $V_b$ for incoming legs),
and putting points $b'\in B'$ on the ``bottom'' $S\times\{0\}$, colored by
objects $V_{b'}$ for outgoing legs.

\begin{figure}[ht]
\begin{tikzpicture}
\node[morphism] (ph) at (0,0) {$\ph$};
\coordinate (top) at (0,1);
\coordinate (bot) at (0,-1);
\coordinate (left) at (-1.2,0);
\coordinate (right) at (1.2,0);
\draw[<-] (ph)-- (-1,0|-top) node[pos=0.7, left] {$V_1$} ;
\draw[<-] (ph)-- (1,0|-top) node[pos=0.7, right] {$V_n$};
\draw[->] (ph)-- (-1,0|-bot) node[pos=0.7, left=2pt] {$V'_1$} ;
\draw[->] (ph)-- (1,0|-bot) node[pos=0.7, right] {$V'_m$};
\draw[boundary] (left|-top)--(right|-top); 
\draw[boundary] (left|-bot)--(right|-bot); 
\end{tikzpicture}
  \caption{Morphisms in $\Chat(S)$}
  \label{f:morphisms_Chat}
\end{figure}
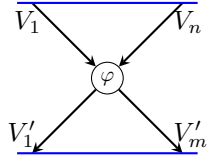

\end{definition}


Note also, if $b, b'\in B$ are such that $b'$ is the successor  of 
$b$ in the order given by orientation of $S$, then we have an isomorphism
$( \dots V_b, V_{b'}, 
\dots) \simeq (\dots V_b\otimes V_{b'}, \dots)$ in  $\Chat$ shown in
\firef{f:chat_isom}. Thus,  the category $\Chat$ meets the requirements
outlined in informal discussion before. 

\begin{figure}[ht]
$$
\begin{tikzpicture}
\node[dotnode] (ph) at (0,1) {};
\draw[->] (ph)-- +(0,-1) node[pos=0.5,right] {$V_b\otimes V_{b'}$};
\draw[<-] (ph)-- (-0.5,2) node[pos=0.5,left] {$V_b$};
\draw[<-] (ph)-- (0.5,2cm) node[pos=0.5,right] {$V_{b'}$};
\draw[boundary] (-1,2)--(1,2) (-1,0)--(1,0);
\end{tikzpicture}
$$
  \caption{}
  \label{f:chat_isom}
\end{figure}

Note that $\Chat$ in general is not an abelian category, even
though the morphisms do form  vector spaces.  Let $\C(S)$ be the
pseudo-abelian completion of $\Chat$  i.e., the  category obtained by
adjoining finite direct sums and images of idempotents. Recall that the
morphisms in this category are defined by  
\begin{equation}\label{e:C_morphisms} 
\Hom_{\C(S)}(Y_1, Y_2)=\{f\in \Hom_{\Chat(S)}(X_1,X_2)\st P_2f=fP_1=f\}
\end{equation}
if $Y_i=\im(P_i)$ for some idempotents $P_i\colon X_i\to X_i$, $ X_i\in \Obj \Chat(S)$. 
\begin{example}\label{x:C(R)}
  Let $S=\R$. Then $\Chat(S)\simeq \C(S)\simeq \A$, with the equivalence
given by $(V_{p_1},\dots, V_{p_n})\mapsto V_{p_1}\otimes \dots\otimes
V_{p_n}$. 
\end{example}

\begin{theorem}\label{t:Chat}
\par\noindent
\begin{enumerate}
 \item Any orientation-preserving homeomorphism of 1-dimensional 
    manifolds $\ph\colon S\to S'$ 
    gives rise to an equivalence $\ph_*\colon \C(S)\to \C(S')$ such 
    that $(\ph \psi)_*=\ph_*\psi_*$  \textup{(}note that it is an equality and 
    not just an isomorphism\textup{)}. Moreover, every homotopy $\ph_t$ 
    between two homeomorphisms gives rise to a isomorphism of functors 
    $\ph_0\to \ph_1$. 
         
 \item One has a natural equivalence 
      $\C(S\sqcup S')\simeq \C(S)\boxtimes \C(S')$.
\end{enumerate}
\end{theorem}
The proof is straightforward.

\begin{theorem}\label{t:drinfeld_center}
 Let $S^1$ be the standard circle. Then one has an 
 equivalence   $J\colon \C(S^1)\simeq Z(\A)$, where $Z(\A)$ is the Drinfeld
center of $\A$ \textup{(}see  \seref{s:drinfeld_center}\textup{)}. 
\end{theorem}

Before proving this theorem, we note the following useful corollary.

\begin{corollary}
For any oriented 1-manifold $S$, $\C(S)$ is an abelian category.
\end{corollary}

Indeed, it immediately follows from the fact that $\C(\R)\simeq \A$ and
$\C(S^1)\simeq Z(\A)$ are abelian (\exref{x:C(R)},
\thref{t:drinfeld_center}) and \thref{t:Chat}. 

\begin{proof}[Proof of \thref{t:drinfeld_center}]

The proof uses several results about Drinfeld center of $\A$, which 
are collected in \seref{s:drinfeld_center}. In particular, we will 
denote by $F\colon Z(\A)\to \A$ the forgetful functor and by 
$I\colon \A\to Z(\A)$ the adjoint functor. Explicit description 
of this functor can be found in \seref{s:drinfeld_center}.

As before, we will frequently use graphical presentation of morphisms
which involve objects both of $\A$ and $Z(\A)$. In these diagrams, we
will show objects of $Z(\A)$ by double green lines   and the
half-braiding isomorphism $\ph_Y\colon Y\otimes V\to V\otimes Y$  by
crossing as in \firef{f:crossing}.
\begin{figure}[ht]
\begin{tikzpicture}[baseline=0pt]
\draw[drinfeld center](0,2)--(1,0) node [pos=0.2, left,black] {$Y$};
\draw[overline](1,2) -- (0,0) node [pos=0.2, right] {$V$};
\end{tikzpicture}

\caption{Graphical presentation of the half-braiding $\ph_Y\colon Y\otimes
V\to V\otimes Y$, $Y\in \Obj Z(\A)$, $V\in \Obj \A$}\label{f:crossing}
\end{figure}
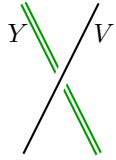

Let $\C'$ be the full subcategory in $\Chat(S^1)$ formed by objects
$\VV=(B, \{V_b\})$ such that $B$ does not contain the point $p_0=(1,0)\in S^1$.
Since it is obvious that every object in $\Chat(S^1)$ is isomorphic to an
object in $\C'$, the inclusion $\C'\subset \C$ is an equivalence.

Let us now construct the functor $J\colon \C'\to Z(\A)$ as follows. Let 
$\VV=(B, \{V_b\})$ be an object in $\C'$. Number the points of $B$, writing
$B=\{b_1, \dots, b_k\}$ going counterclockwise
starting with $p_0$. Define
$$
J(\VV)=I(V_1\otimes\dots \otimes V_k)
$$
where $I\colon A\to Z(\A)$ is the functor \eqref{e:I}. 

Now, define $J$ on morphisms as follows. Let $\Ga\in \Hs(S^1\times I)$ be a
colored graph representing a morphism $\VV\to \VV'$. Without loss of
generality, we can assume that this graph does not have any vertices on the
interval $\{p_0\}\times I$. Define $J(\Ga)\colon J(\VV)\to J(\VV')$ be the
morphism represented by the graph $\Ga''\subset [0,1]\times I$, obtained by
\begin{enumerate}
  \item Replacing all edges of $\Ga$ crossing $\{p_0\}\times I$ by a single
edge, colored by a linear combination of simple objects of $\A$, as in the
proof of \leref{l:main2}.
  \item Cutting the cylinder along the interval $\{p_0\}\times I$, to get
    a colored graph in $[0,1]\times I$
  \item Adding to the obtained graph four new legs and two vertices
  colored as in \firef{f:JGa}.
\end{enumerate}
\begin{figure}[ht]
\begin{align*}
&
\begin{tikzpicture}[yscale=0.5]
\coordinate (t1) at (-1,1.5); \node[left] at (t1) {$\VV$};
\coordinate (t4) at (1, 1.5);
\coordinate (b1) at (-1,-1.5);\node[left] at (b1) {$\VV'$};
\coordinate (b4) at (1, -1.5);
\coordinate (t5) at (0, 0.5);
\coordinate (b5) at (0, -2.5);
\path (t1) arc (-180:-135:1cm) node (t2) {};
\fill[subgraph] (t1) arc (180:0:1cm) -- +(0,-3) arc (0:-180:1cm)--(t1);
\fill[white] (t2) arc (-135:-45:1cm)-- +(0,-3) arc
(-45:-135:1cm)--(t2);
\node[above] at (t5) {$p_0$};
\draw[boundary] (t1) arc (-180:0:1cm) (b4) arc (0:-180:1cm);
\draw[thin] (t1)--(b1) (t4)--(b4);
\draw[boundary] (t1) arc (180:0:1cm);
\draw[boundary,densely dotted,thin] (b1) arc (180:0:1cm);
\draw[edge] (t5)--(b5);
\draw (t2) ++(0,-1.6) arc (-135:-45:1cm);
\draw (t2) ++(0,-1.2) arc (-135:-45:1cm);
\draw (t2) ++(0,-1.4) arc (-135:-45:1cm);
\end{tikzpicture}
\qquad\to \qquad 
\begin{tikzpicture}[yscale=0.5]
\coordinate (t1) at (-1,1.5); \node[left] at (t1) {$\VV$};
\coordinate (t4) at (1, 1.5);
\coordinate (b1) at (-1,-1.5);\node[left] at (b1) {$\VV'$};
\coordinate (b4) at (1, -1.5);
\coordinate (t5) at (0, 0.5);
\coordinate (b5) at (0, -2.5);
\path (t1) arc (-180:-135:1cm) node (t2) {};
\fill[subgraph] (t1) arc (180:0:1cm) -- +(0,-3) arc (0:-180:1cm)--(t1);
\fill[white] (t2) arc (-135:-45:1cm)-- +(0,-3) arc
(-45:-135:1cm)--(t2);
\node[above] at (t5) {$p_0$};
\draw[boundary] (t1) arc (-180:0:1cm) (b4) arc (0:-180:1cm);
\draw[thin] (t1)--(b1) (t4)--(b4);
\draw[boundary] (t1) arc (180:0:1cm);
\draw[boundary,densely dotted,thin] (b1) arc (180:0:1cm);
\draw[edge] (t5)--(b5);
\node[small_morphism] (al1) at (-0.3,-0.8) {$\al$};
\node[small_morphism] (al2) at (0.3,-0.8) {$\al$};
\draw[regular] (al1)--(al2);
\draw (al1)-- +(-0.4,0.3) (al1)-- +(-0.4,0) (al1)-- +(-0.4,-0.3);
\draw (al2)-- +(0.4,0.3) (al2)-- +(0.4,0) (al2)-- +(0.4,-0.3);
\end{tikzpicture}
\qquad \to \qquad 
\begin{tikzpicture}
\coordinate (t1) at (-2,1);
\coordinate (t4) at (2,1);
\coordinate (b1) at (-2,-1);
\coordinate (b4) at (2,-1);
\fill[subgraph] (-1,-1) rectangle (1,1);
\draw[boundary] (t1) -- (t4)--(b4)--(b1)--(t1);
\node at (0,0) {$\Ga$};
\node[small_morphism] (al1) at (-1.5,0) {$\al$};
\node[small_morphism] (al2) at (1.5,0) {$\al$};
\draw (al1)--(-1,0.3) (al1)--(-1,0) (al1)--(-1,-0.3);
\draw[regular] (al1)--(al1-|t1);
\draw[regular] (al2)--(al2-|t4);
\draw (al2)--(1,0.3) (al2)--(1,0) (al2)--(1,-0.3);
\node[above] at (0,1) {$\VV$};
\node[below] at (0,-1) {$\VV'$};
\end{tikzpicture}
\\  & \qquad \to \qquad 
\sum_{i,j,k\in \Irr(\A)}d_i \sqrt{d_k}\sqrt{d_j} \qquad
\begin{tikzpicture}
\coordinate (t1) at (-2.2,1);
\coordinate (t4) at (2.2,1);
\coordinate (b1) at (-2.2,-1);
\coordinate (b4) at (2.2,-1);
\fill[subgraph] (-1,-1) rectangle (1,1);
\draw[boundary] (t1) -- (t4) (b4)--(b1);
\node at (0,0) {$\Ga$};
\node[small_morphism] (al1) at (-1.5,0) {$\al$};
\node[small_morphism] (al2) at (1.5,0) {$\al$};
\draw (al1)--(-1,0.3) (al1)--(-1,0) (al1)--(-1,-0.3);
\draw (al2)--(1,0.3) (al2)--(1,0) (al2)--(1,-0.3);
\node[small_morphism] (be1) at (-2,0) {$\be$};
\node[small_morphism] (be2) at (2,0) {$\be$};
\draw (al1)--(be1) node[midway,above,scale=0.8] {$i$};
\draw (al2)--(be2) node[midway,above,scale=0.8] {$i^*$};
\draw[->] (be1|-b1)--(be1) node[midway,left,scale=0.8] {$j$};
\draw[<-] (be2|-b4)--(be2) node[midway,right,scale=0.8] {$j$};
\draw[<-] (be1|-t1)--(be1) node[midway,left,scale=0.8] {$k$};
\draw[->] (be2|-t4)--(be2) node[midway,right,scale=0.8] {$k$};
\node[above] at (0,1) {$\VV$};
\node[below] at (0,-1) {$\VV'$};
\end{tikzpicture}
\end{align*}
\caption{}\label{f:JGa}
\end{figure}
The same argument as in the proof of \thref{t:I} shows that $J(\Ga)$ only
depends on the class $\<\Ga\>$ in $\Hs(S^1\times I)$. By
\leref{l:isomorphisms}, $J(\Ga)\in \Hom_{Z(\A)}(J(\VV),J(\VV'))$. It is also
easy to see that so defined $J$ preserves composition of morphisms; thus,
we have defined a functor $J\colon \C'\to Z(\A)$. 

It is obvious from the definition that $J$ is additive and that its
essential image consists of objects of the form $I(V)$, $V\in \Obj(\A)$.
Since $Z(\A)$ is an abelian category, it is clear that this functor extends
naturally to the pseudo-abelian completion of $\C'$ which is equivalent to
the pseudo-abelian completion  $\C$ of $\Chat(S^1)$. Since every object in
$Z(\A)$ is a direct summand of some object of the form $I(V)$, $V\in \A$ 
(see remark after \leref{l:proj_Y}), we see that the functor 
$J\colon \C\to Z(\A)$ is  essentially surjective.

To prove that $J$ is an equivalence, we construct the inverse functor
$K\colon Z(\A)\to \C'$, which sends an object  $Y\in Z(\A)$ to the image of
the projector $P\in \Hom_{\C(S^1)} (\bf{Y}, \bf{Y})$ shown in \firef{f:P}; 
the image makes sense in the pseudo-abelian completion. 
\begin{figure}[ht]
$$
\frac{1}{\DD^2}\qquad
\begin{tikzpicture}[yscale=0.5]
\coordinate (t1) at (-1,1.5);
\coordinate (t4) at (1, 1.5);
\coordinate (b1) at (-1,-1.5);
\coordinate (b4) at (1, -1.5);
\coordinate (t5) at (0, 0.5);
\coordinate (b5) at (0, -2.5);
\path (t4) arc (0:-135:1cm) coordinate (t3) {};
\draw[regular] (-1,0) arc (180:0:1cm);
\draw[drinfeld center] (t3)-- ++(0,-3);
\node[above] at (t3) {$Y$};
\draw[regular,overline](-1,0) arc (-180:0:1cm);
\draw[thin] (t1) arc (-180:0:1cm)-- (b4) arc (0:-180:1cm) -- (t1);
\draw[thin] (t1) arc (180:0:1cm);
\draw[densely dotted,thin] (b1) arc (180:0:1cm);
\end{tikzpicture}
$$
\caption{The projector $P\in \Hom_{\C(S^1)} (\bf{Y}, \bf{Y})$. Here $\bf{Y}$ is $F(Y)$ placed at the point $(-1,0)\in S^1$.}
\label{f:P}
\end{figure}
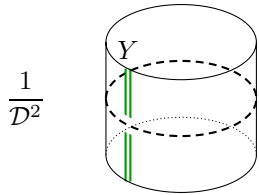
Using \leref{l:proj_Y}, it is easy to show that one has canonical
functorial isomorphisms $KJ\simeq \id$, $JK\simeq\id$. Thus, $J$ is an
equivalence. 
\end{proof}

Note that the explicit construction of the equivalence also gives the following result which we will use later.
\begin{corollary}\label{c:P}
Let  $K\colon Z(\A)\to \C(S^1)$ be the equivalence constructed in 
\thref{t:drinfeld_center}. Then for any $\VV\in \Chat(S^1)$, $Y\in Z(\A)$, we have
$$
\Hom_{\C(S^1)}(\VV, K(Y))=\{f\in \Hs(S^1\times I, \VV^*, {\bf Y})\st Pf=f\}
$$
where $\bf Y$, $P$ are as shown in \firef{f:P}. 
\end{corollary}

\section{Extended theory: excited states}\label{s:excited}

Let  now $\Si$ be an oriented surface with boundary. Recall that then,
for every choice of boundary conditions $\VV\in \Obj \Chat (\del \Si)$, we
have the vector space of string nets  $\Hs(\Si;\VV)$ (see 
\deref{d:string-net}). We can extend this to the pseudo-abelian 
completion $\C(\del \Si)$ as follows. 

\begin{definition}\label{d:strings_boundary}
  Let $\Si$ be an oriented surface (possibly with boundary) and 
  $Y\in \Obj \C (\del \Si)$. Define 
$$ 
\Hs(\Si,Y)=\VGr(\Si,Y)/N
$$
where 
\begin{itemize}
\item $\VGr(\Si,Y)$ is the vector space of formal linear combinations of
pairs 
  $\hat\Ga=(\ph, \Ga)$, where $\Ga$ is a colored graph on $\Si$ with 
  some boundary value $\VV$ and $\ph\in\Hom_{\C(\del \Si)}(\VV, Y)$
  
\item $N$ is the subspace of local relations, spanned by the same local
relations as in \deref{d:string-net}  (coming from embedded disks $D\subset
\Si$) and additional relation 
  \begin{equation}\label{e:strings_boundary}
    (\ph f, \Ga)=(\ph, f\Ga)
  \end{equation}
where $\Ga$ is a colored graph with boundary value $\VV$ and 
$f\in \Hom_{\C(\del \Si)}(\VV, \VV')=\Hs(\del \Si\times I, \VV^*, \VV')$,
$\ph\in \Hom_{\C(\del \Si)}(\VV', Y)$. Here $f\Ga$ means the graph obtained by composing $\Ga$ and $f$. 
\end{itemize}
\end{definition}
It is immediate from the definition that for $Y\in \hat \C(\del \Si)$, this
definition coincides with \deref{d:string-net}; it is also obvious that
this definition is functorial in $Y$.

Note that if we  choose, for every boundary circle $(\del \Si)_a$ of $\Si$,
an orientation-preserving homeomorphism $\psi_a\colon (\del \Si)_a\to
S^1$, then by  \thref{t:drinfeld_center}, \thref{t:Chat}, this gives rise
to an equivalence 
\begin{equation}\label{e:C}
K_\psi\colon  Z(\A)^{\boxtimes A}\simeq \C(\del \Si)
\end{equation}
 where $A$ is the set of boundary components of $\del \Si$. Thus, given a
collection of objects $Y_a\in Z(\A)$, $a\in A$, we can define the space 
\begin{equation}
 \Hs(\Si, \{\psi_a\}, \{Y_a\})=\Hs(\Si, K_\psi(\boxtimes Y_a))
\end{equation}

The space $\Hs(\Si, \{\psi_a\}, \{Y_a\})$ admits an alternative definition.
Namely, let $\Sihat$ be the closed surface obtained by gluing to $\Si$ a
copy of the standard 2-disk $D$ along each boundary circle $(\del \Si)_a$
of $\Si$, using parametrization $\psi_a$. So defined surface comes with a
collection of marked points $p_a=\psi_a^{-1}(p)$, where $p=(1,0)$ is the
marked point on $S^1$. Moreover, for every  point $p_a$ we also have a
distinguished  ``tangent direction'' $v_a$ at $p_a$ (in PL setting, we
understand it as a germ of an arc staring at $p_a$), namely the direction
of the radius connecting $p$ with the center of the disk $D$. We will refer
to the collection  $(\Sihat, \{p_a\}, \{v_a\})$ as an {\em extended
surface}. It is easy to see that given  $(\Sihat, \{p_a\}, \{v_a\})$, the
original surface $\Si$ and parametrizations $\psi_a$ are defined uniquely
up to a contractible set of choices.

Given such an extended surface $\Sihat$ and a colection of objects $Y_a\in
Z(\A)$, one object for each marked point $p_a$, define the string net space
\begin{equation}\label{e:hhs}
\Hhs(\Sihat, \{p_a\}, \{v_a\}, \{Y_a\})=
\VGr'(\Sihat, \{p_a\}, \{v_a\}, \{Y_a\})/N
\end{equation}
where 
\begin{itemize}
\item $\VGr'(\Sihat, \{p_a\}, \{v_a\}, \{Y_a\})$ is the vector space of
formal linear combinations of colored graphs on $\Sihat$ such that each
colored graph has an uncolored one-valent vertex at each point $p_a$, with
the corresponding edge coming from direction $v_a$ (i.e., in some
neighborhood of $p_a$, the edge coincides with the corresponding arc) and
colored by the object $F(Y_a)$ as shown in  \firef{f:marked_point}.
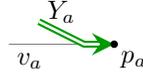
\begin{figure}[ht]
$$
\begin{tikzpicture}
\draw[thin, gray] (-1.4,0)--(0,0) node[pos=0.2, below, black] {$v_a$};
\draw[drinfeld center, ->] (-1,0.3)--(-0.4,0) node[pos=0.5, above] {$Y_a$}
                               --(0,0);
\node[dotnode, label=-45:$p_a$] at (0,0) {};
\end{tikzpicture}
$$
\caption{Colored graphs in a neigborhood of marked point}
\label{f:marked_point}
\end{figure}

\item $N$ is the subspace of local relations, spanned by the same local
relations as in \deref{d:string-net}  coming from embedded disks $D\subset
\Si$ not containing the special points $p_a$ and additional local relations
in a neighborhood of each marked point  $p_a$ shown in
\firef{f:local_rels_marked}. 

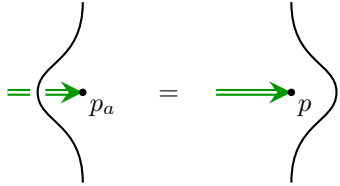
\begin{figure}[ht]
$$
\begin{tikzpicture}
\draw[drinfeld center, ->] (-1,0)--(0,0);
\node[dotnode, label=-45:$p_a$] at (0,0) {};
   \draw[overline] (0,1.2)..controls +(-90:0.8cm) and +(90:0.4cm) ..
                   (-0.6, 0) ..controls +(-90:0.4cm) and +(90:0.8cm) ..
                   (0,-1.2);
 \end{tikzpicture}
\quad=\quad
  \begin{tikzpicture}
\draw[drinfeld center, ->] (-1,0)--(0,0);
\node[dotnode, label=-45:$p$] at (0,0) {};
    \draw (0,1.2)..controls +(-90:0.8cm) and +(90:0.4cm) ..
                   (0.6, 0) ..controls +(-90:0.4cm) and +(90:0.8cm) ..
                   (0,-1.2);
  \end{tikzpicture}
$$
\caption{Extra local relation near  marked point}\label{f:local_rels_marked}
\end{figure}

\end{itemize}

\begin{theorem}\label{t:hhs}
  Let $\Si$, $\Sihat$ be as above, and let $Y_a$, $a\in A$, be  a collection of
  objects in $Z(\A)$, one object for each boundary component of $\Si$. 
Then one has a canonical isomorphism  
  \begin{equation}
  \Hs(\Si, \{\psi_a\}, \{Y_a\})\simeq \Hhs(\Sihat, \{p_a\}, \{v_a\}, \{Y_a\})
  \end{equation}
  
  \end{theorem}
   \begin{proof}
   Denote for brevity
   \begin{align*}
   \Hs&=\Hs(\Si, \{\psi_a\}, \{Y_a\})\\
   \Hhs&=\Hhs(\Sihat, \{p_a\}, \{v_a\}, \{Y_a\})
   \end{align*}
   
    By definition, the space $\Hs$ is defined as
    the vector space of pairs $(\{\ph_a\}, \Ga)$ modulo local relations;
    here $\Ga$ is a colored graph on $\Si$ with boundary value
    $\VV=\{\VV_a\}, a\in A$, and   
    $$
    \ph_a\in \Hom_{\C(S^1)} (\VV_a, K(Y_a))=    
    \{f\in \Hom_{\Chat(S^1)} (\VV_a, {\bf Y_a})\st Pf=f\}
    $$ 
    where $\bf Y$, $P$ are as in \coref{c:P}. 
   
    Construct the map 
    $T\colon \Hs\to \Hhs$ by 
    $$
    T(\{\ph_a\}, \Ga)=\Ga'
    $$
    where the graph $\Ga'$ is given by \firef{f:hhs1} in the neighborhood 
    of the glued disk $D_a$ and coincides with $\Ga$ elsewhere.
\begin{figure}[ht]
$$
\begin{tikzpicture}
\node[dotnode, label=-45:$p_a$] (p) at (0.5,0) {};
\draw (60:0.5cm) -- +(60:0.4cm);
\draw (120:0.5cm) -- +(1200:0.4cm);
\draw (200:0.5cm) -- +(200:0.4cm);
\draw (-50:0.5cm) -- +(-50:0.4cm);
\draw[boundary] (0,0) circle(0.5cm);
\node at (-0.2,1.2) {$\Ga$};
 \end{tikzpicture}
\mapsto 
\begin{tikzpicture}
\fill[even odd rule,subgraph] (0,0) circle (.6cm) (0,0) circle (1.6cm);
\node (ph) at (-1,0) {$\ph_a$};
\draw[drinfeld center] (-0.6,0)--(0.5,0) node[pos=0.5,above] {$Y_a$};
\node[dotnode, label=-45:$p_a$] (p) at (0.5,0) {};
\draw[edge] (0,0) circle(0.5cm);
\draw (60:1.6cm) -- +(60:0.4cm);
\draw (120:1.6cm) -- +(1200:0.4cm);
\draw (200:1.6cm) -- +(200:0.4cm);
\draw (-50:1.6cm) -- +(-50:0.4cm);
 \end{tikzpicture}
$$
\caption{Map $T\colon \Hs\to \Hhs$. The shaded area represents $\ph_a\in \Hom_{\C(S^1)} (\VV_a, K(Y_a))$.}\label{f:hhs1}
\end{figure}

    It is immediate from the definition that this map is well defined and
surjective. To prove that it is actually an isomorphism, construct now the 
map  
    \begin{align*}
    S\colon \Hhs(\Sihat, \{p_a\}, \{v_a\}, \{Y_a\})&\to 
                 \Hs(\Si, \{\psi_a\}, \{Y_a\}\\
    \Ga &\mapsto (\{\id\}, \Ga')
    \end{align*}
    where $\Ga'$ is given by \firef{f:hhs2} in the neighborhood 
    of the boundary component $(\del \Si)_a$ and coincides with $\Ga$ elsewhere.    
\begin{figure}[ht]
$$
\begin{tikzpicture}
\draw[drinfeld center] (-0.8,0)--(0.5,0) node[pos=0.5,above] {$Y_a$};
\node[dotnode, label=-45:$p_a$] (p) at (0.5,0) {};
\draw[edge] (0,0) circle(0.5cm);
 \end{tikzpicture}
\mapsto
 \begin{tikzpicture}
\draw[drinfeld center] (-0.8,0)--(-0.5,0) node[pos=0,above] {$Y_a$};
\node[dotnode, label=-45:$p_a$] (p) at (0.5,0) {};
\draw[boundary] (0,0) circle(0.5cm);
 \end{tikzpicture}
$$
\caption{Map $S\colon \Hhs\to \Hs$}\label{f:hhs2}
\end{figure}

    It is easily checked that the map $S$ is well defined and is inverse to
$T$. This completes the construction of the isomorphism.  
\end{proof}

We can now prove the second main result of the paper, extending
\thref{t:main} to the case of surfaces with boundary. Recall that one can
extend Turaev--Viro theory, defining vector spaces $Z_{TV}(\Si,
\{Y_a\}_{a\in A})$ for a surface $\Si$ with boundary together with a choice
of marked point on each boundary component and a choice of an object
$Y_a\in Z(\A)$ for each boundary component (see \ocite{balsam-kirillov}).
Since a choice of a parameterization $\psi_a\colon (\del \Si)_a\to S^1$
also determines a marked point $p_a=\psi_a^{-1}(p)$, $p=(1,0)\in S^1$, we
can also define the space $Z_{TV}(\Si, \{Y_a\})$ for a surface with a
parametrized boundary.   

\begin{theorem}\label{t:main2}
Let $\Si$ be a compact oriented  surface  with boundary together with a
choice of a parameterization $\psi_a\colon (\del \Si)_a\to S^1$ for  each
boundary component and a choice of an object $Y_a\in Z(\A)$ for each
boundary component. Then one has a canonical isomorphism  
$$
Z_{TV}(\Si, \{Y^*_a\}_{a\in A})\simeq \Hs(\Si, \{Y_a\})
$$
where $A$ is the set of connected components of the boundary of $\Si$.
\end{theorem}
\begin{proof}

The proof repeats with necessary changes the proof of \thref{t:main}. We
outline the main steps below, stressing the changes. 

First, we need to choose a cell decomposition $\De$ of $\Si$ such that for
each boundary circle, the marked point $p_a$ is one of the vertices of
$\De$. This  this also gives a cell decomposition $\hat \De$ of the surface
$\Sihat$ obtained by gluing a disk to every boundary component of $\Si$
(see \thref{t:hhs}). As in \ocite{balsam-kirillov}, this allows us to
define the vector space $\HTV(\Sihat, \{Y^*_a\})$. As before, we let
$\De^0$ be the set of all vertices of $\De$.

Define now the string net space $\HhsD(\Sihat, \{Y_a\})$ as the vector
space of colored graphs $\Ga\in \Sihat-\De^0$ such that in  a neighborhood
of each marked point $p_a$ the graph looks as  shown in
\firef{f:marked_point},  modulo the same local relations as in
\deref{d:string-net}, for any embedded disk $D\subset \Sihat-\De^0$ (note
that we do not impose extra local relation  of
\firef{f:local_rels_marked}). 

Then we have the following results:

\begin{enumerate}
\item Define, for every point $p\in \De^0$, the operator $B_p$ by
\firef{f:B_p_a} if $p=p_a$ is a marked point and  by the same formula as in
\thref{t:B_p} for other vertices. 
\begin{figure}[ht]
$$
\begin{tikzpicture}
\draw[drinfeld center] (-1.2,0)--(0,0) node[pos=0.2, above] {$Y_a$};
\node[dotnode, label=45:$p_a$] at (0,0) {};
\end{tikzpicture}
\qquad\mapsto\quad
\frac{1}{\DD^2}\quad
\begin{tikzpicture}
\draw[drinfeld center] (-1.2,0)--(0,0) node[pos=0.2, above] {$Y_a$};
\node[dotnode, label=45:$p_a$] at (0,0) {};
\draw[overline,regular] (0,0) circle(0.6cm);
\end{tikzpicture}
$$
\caption{Operator $B_p$ for the marked point $p=p_a$.}\label{f:B_p_a}
\end{figure}

Then operators $B_p$ are mutually commuting projectors, and we have a
natural isomorphism  
$$
\Hs(\Si, \{Y_a\})\simeq \Hhs(\Sihat, \{Y_a\})=\im(B)
$$
where 
$$
B=\prod_{p\in \De^0} B_p\colon \HhsD(\Sihat, \{Y_a\})\to \HhsD(\Sihat, \{Y_a\})
$$
(compare with \thref{t:B_p}).

The proof of this result is quite similar to the proof of  \thref{t:B_p};
details are left to the reader. 

\item One has a  natural isomorphism $\HhsD(\Sihat, \{Y_a\})\simeq
\HTV(\Si, \De, \{Y^*\})$ 

The isomorphism is constructed in the same way as in \leref{l:main2}, with the only change that for the glued disk $D_a$ (which is a cell of the decomposition $\hat\De$), we add an extra edge to the dual graph $\Ga_\De$ and coloring of the vertex is as shown in \firef{f:dual_graph2}.
\begin{figure}[h]
\begin{tikzpicture}
\node[small_morphism] (O) at (0,0) {$\ph$};
\node[dotnode](p) at (0:1.2cm) {};
\path (p) node[right] {$p_a$};
\draw[drinfeld center, <-] (O) --(p);\foreach \i in {1,...,5}
   {
    \pgfmathsetmacro{\u}{(\i-1)*72}
    \pgfmathsetmacro{\v}{\u+72}
    \pgfmathsetmacro{\w}{\u+36}
    \draw[edge, ->] (\u:1.2cm)--(\v:1.2cm);
    \draw[->] (O)--(\w:2cm)  node[pos=0.7,auto=left]{$X_\i$};
    }

\end{tikzpicture}
\qquad 
$\ph\in \<F(Y_a)^*,X_1,\dots, X_5\>=\Hom_\A(F(Y_a), X_1\otimes\dots\otimes X_5)$
\caption{Coloring of the dual graph for the embedded disk.}\label{f:dual_graph2}
\end{figure}
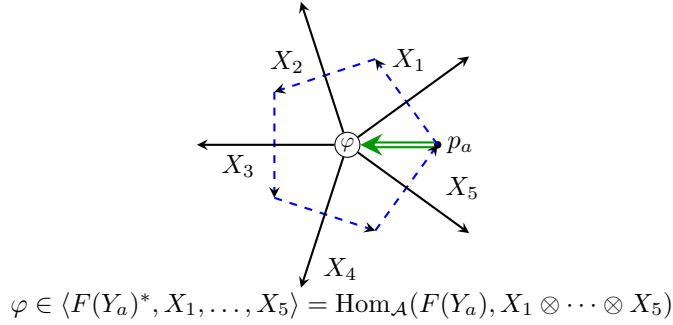

\item  The isomorphism of the previous part identifies the operator $B$
with $\ZTV(\Si\times I)\colon H_{TV}(\Si, \De,\{Y^*\})\to H_{TV}(\Si,
\De,\{Y^*\})$ 

This is proved in the same way as \leref{l:main3}

\end{enumerate} 
\end{proof}

\section{Drinfeld center: technical lemmas}
\label{s:drinfeld_center}

In this section we collect some basic facts about the Drinfeld center of a fusion category, which were used in the proof of \thref{t:drinfeld_center}. Throughout this section, $\A$ be a spherical fusion category over an algebraically closed  field of characteristic zero.

Recall that the Drinfeld center $Z(\A)$ of a fusion category $\A$
is defined as the category whose objects are pairs $(Y,\ph_Y)$, where $Y$
is an object of $\A$ and $\ph_Y$ -- a functorial isomorphism $Y\otimes -\to
-\otimes Y$ satisfying certain compatibility conditions (see
\ocite{muger1}). We will refer to  $\ph_Y$ as ``half-braiding''.

\begin{theorem}{\ocite{muger2}}
  Let $\A$ be a spherical fusion category over an algebraically closed   
  field of characteristic zero.Then  $Z(\A)$ is a modular category; in 
  particular, it is semisimple with finitely many simple objects, it is 
  braided and has a pivotal structure which coincides with the pivotal 
  structure on $\A$. 
\end{theorem}

We have an obvious forgetful functor $F\colon Z(\A)\to \A$. To simplify
the notation, we will frequently omit it in the formulas, writing for
example $\Hom_\A(Y,V)$ instead of $\Hom_\A(F(Y),V)$, for $Y\in \Obj
Z(\C)$, $V\in \Obj \A$. Note, however, that if $Y,Z\in \Obj Z(\A)$, then
$\Hom_{Z(\A)}(Y,Z)$ is different from $\Hom_{\A}(Y,Z)$: namely,
$\Hom_{Z(\A)}(Y,Z)$ is a subspace in $\Hom_{\A}(Y,Z)$ consisting of those
morphisms that commute the with the half-braiding.

The following theorem, proof of which can be found in
\ocite{balsam-kirillov}, gives an explicit description of the functor
$I\colon \A\to Z(\A)$ adjoint to $F$. 
\begin{theorem}\label{t:I}
   Define, for   $V\in \Obj \A$
  \begin{equation}\label{e:I}
    I(V)=\bigoplus_{i\in \Irr(\A)} X_i\otimes V\otimes X_i^*.
  \end{equation}
  Then
  \begin{enumerate}
   \item $I(V)$ has a natural structure of an object of $Z(\A)$, with 
     the half braiding given by \firef{f:I}.
     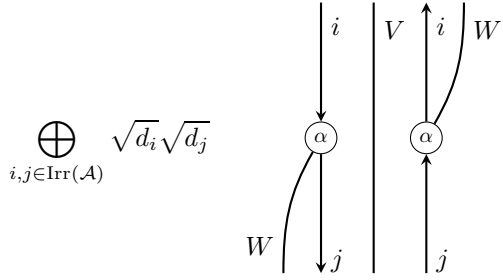
\begin{figure}[ht]
     $$\bigoplus_{i,j\in \Irr (\A)} \sqrt{d_i}\sqrt{d_j} \quad
\begin{tikzpicture}
\coordinate (bot) at (0,-1.8);
\coordinate (top) at (0,1.8);
\draw (bot)--(top) node[pos=0.9, right] {$V$};
\node[morphism] (ph) at (-0.7,0) {$\al$};
\node[morphism] (ph*) at (0.7,0) {$\al$};
\draw[<-] (ph|-bot) -- (ph) node[pos=0.1,right] {$j$};
\draw[<-] (ph) -- (ph|-top) node[pos=0.8,right] {$i$};
\draw[out=-120, in=90] (ph.-120) to (-1.2,-1.8) node[pos=0.8,left] {$W$};
\draw[->] (ph*|-bot) -- (ph*) node[pos=0.1,right] {$j$};
\draw[->] (ph*) -- (ph*|-top) node[pos=0.8,right] {$i$};
\draw[out=60, in=-90] (ph*.60) to (1.2,1.8) node[pos=0.8,right] {$W$};
\end{tikzpicture}
     $$
     \caption{Half-braiding $I(V)\otimes W\to W\otimes I(V)$.}
     \label{f:I}
     \end{figure}

   \item So defined functor $I\colon \A\to Z(\A)$ is two-sided adjoint of
the forgetful functor $F\colon Z(\A) \to \A$: one has functorial
isomorphisms
    \begin{equation}
    \begin{aligned}
     \Hom_\A(V,F(X))&\simeq \Hom_{Z(\A)}(I(V),X)\\
          \ph&\mapsto \sum_{i\in \Irr(\A)}\frac{\sqrt{d_i}}{\DD}\quad
\begin{tikzpicture}
\node[morphism] (ph) at (0,0) {$\ph$};
\draw[drinfeld center] (0,-1)--(ph); \draw (ph)--(0,1);
\draw[overline] (-0.5,1)--(-0.5,0) arc(-180:0:0.5cm) --(0.5,1) node[right]
{$i$};
\end{tikzpicture}
    \end{aligned}
    \end{equation}
and 
    \begin{equation}
    \begin{aligned}
     \Hom_\A(F(X), V)&\simeq \Hom_{Z(\A)}(X,I(V))\\
          \ph&\mapsto \sum_{i\in \Irr(\A)} \frac{\sqrt{d_i}}{\DD}\quad
\begin{tikzpicture}
\node[morphism] (ph) at (0,0) {$\ph$};
\draw (0,-1)--(ph)node[pos=0,below]{$V$}; 
\draw[drinfeld center] (ph)--(0,1) node[above] {$X$};
\draw[overline] (-0.5,-1)--(-0.5,0) arc(180:0:0.5cm) --(0.5,-1)
node[right] {$i$};
\end{tikzpicture}
    \end{aligned}
    \end{equation}
\end{enumerate}
\end{theorem}

\begin{lemma}\label{l:proj_Y}
Let $Y\in \Obj Z(\A)$. Define $P_Y\in \End_\A(I(Y),I(Y))$ by
$$
P_Y=\sum_{i,j \in \Irr(\A)} \frac{\sqrt{d_i}\sqrt{d_j}}{\DD^2}
\begin{tikzpicture}
\draw[drinfeld center] (0,-1)--(0,1) node[above] {$Y$};
\draw[overline] (-0.5,1) arc(-180:0:0.5cm) --(0.5,1)
node[right] {$i$};
\draw[overline] (-0.5,-1) arc(180:0:0.5cm) --(0.5,-1)
node[right] {$j$};
\end{tikzpicture}
$$
Then
\begin{enumerate}
  \item $P_Y\in \End_{Z(\A)}(I(Y))$
  \item $P_Y^2=P_Y$
  \item The image of $P_Y$ is canonically isomorphic to $Y$ as an object
of $Z(\A)$.
\end{enumerate}
\end{lemma}
\begin{proof}
  Easily follows from \thref{t:I}.
\end{proof}
In particular, this  lemma implies that every object $Y$ in $Z(\A)$ is a
direct summand of an object of the form $I(V)$ for some $V\in \A$ (suffices
to take $V=F(Y)$).

We will need one more lemma.
\begin{lemma}\label{l:isomorphisms}
  For any $V,W\in \Obj(\A)$, one has the following commutative diagram of
functorial isomorphisms
\begin{equation}
   \begin{tikzpicture}
     \node[anchor=base] (a) at (3,0)
        {$\displaystyle \bigoplus_{Z\in\Irr(Z(\A))}
                   \Hom_\A(V,Z)\otimes \Hom_\A(Z, W)$};
     \node[anchor=base] (b) at (0,-2)
               {$\displaystyle \bigoplus_{i\in \Irr(\A)}
                \Hom_\A(V, i\otimes W\otimes i^*)$};
    \node[anchor=base] (c) at (6, -2)
         {$\displaystyle \Hom_{Z(\A)}(I(V),I(W))$};
   \draw[thin,->] (a)--(b) node[pos=0.5, left=3pt] {$f_1$};
   \draw[thin,->] (a)--(c) node[pos=0.5, right=3pt] {$f_2$};
   \draw[thin,->] (b)--(c) node[pos=0.5, above] {$f_3$};
    \end{tikzpicture}
\end{equation}
where the maps $f_i$ are defined by
\begin{align*}
 f_1&\colon \ph\otimes\psi\mapsto\sum_{i\in \Irr(\A)}\frac{\sqrt{d_i}}{\DD}
\begin{tikzpicture}
\node[morphism] (ph) at (0,0.5) {$\ph$};
\node[morphism] (psi) at (0,-0.5) {$\psi$};
\draw[drinfeld center] (ph)--(psi);
\draw (ph)--(0,1.2) node[above]{$V$};
\draw (psi)--(0,-1.2) node[below]{$W$};
\draw[overline] (-0.5,-1.2)--(-0.5,-0.5) arc(180:0:0.5cm) --(0.5,-1.2)
node[right] {$i$};
\end{tikzpicture}
\\
f_2&\colon \ph\otimes\psi\mapsto
\sum_{i,j \in \Irr(\A)} \frac{\sqrt{d_i}\sqrt{d_j}}{\DD^2}
\begin{tikzpicture}
\node[morphism] (psi) at (0,-0.8) {$\psi$};
\node[morphism] (ph) at (0,0.8) {$\ph$};
\draw[drinfeld center] (ph)--(psi);
\draw (ph)--(0,1.5) node[above]{$V$};
\draw (psi)--(0,-1.5) node[below]{$W$};
\draw[overline] (-0.5,1.5)--(-0.5,0.8) arc(-180:0:0.5cm) --(0.5,1.5)
node[right] {$i$};
\draw[overline] (-0.5,-1.5)--(-0.5,-0.8) arc(180:0:0.5cm) --(0.5,-1.5)
node[right] {$j$};
\end{tikzpicture}
\\
f_3&\colon \ph\mapsto
\sum_{j,k \in \Irr(\A)}\frac{\sqrt{d_i}\sqrt{d_j}\sqrt{d_k}}{\DD}\qquad
\begin{tikzpicture}
\node[morphism] (ph) at (0,0) {$\ph$};
\node[morphism] (al1) at (-1,0) {$\al$};
\node[morphism] (al2) at (1,0) {$\al$};
\draw[<-] (al1)--(ph) node[midway, above]{$i$};
\draw[->] (ph)--(al2)  node[midway, above]{$i^*$};
\draw[->] (ph)--(0,-1) node[below]{$W$};
\draw[<-] (ph)--(0,1) node[above]{$V$};
\draw[<-] (al1)--(-1,-1)node[midway, left]{$j$};
\draw[<-] (-1,1)--(al1) node[midway, left]{$k$};
\draw[->] (al2)--(1,-1)node[midway, right]{$j$};
\draw[->] (1,1)--(al2) node[midway, right]{$k$};
\end{tikzpicture}
\end{align*}

\end{lemma}
The proof of this lemma repeats with minor changes the proof of Theorem~7.3
in \ocite{balsam-kirillov}.

%








\begin{bibdiv}
\begin{biblist}


\bib{BK}{book}{
   label={BakK2001},
   author={Bakalov, Bojko},
   author={Kirillov, Alexander, Jr.},
   title={Lectures on tensor categories and modular functors},
   series={University Lecture Series},
   volume={21},
   publisher={American Mathematical Society},
   place={Providence, RI},
   date={2001},
   pages={x+221},
   isbn={0-8218-2686-7},
   review={\MR{1797619 (2002d:18003)}},
}
\bib{balmer}{article}{
   author={Balmer, Paul},
   author={Schlichting, Marco},
   title={Idempotent completion of triangulated categories},
   journal={J. Algebra},
   volume={236},
   date={2001},
   number={2},
   pages={819--834},
   issn={0021-8693},
   review={\MR{1813503 (2002a:18013)}},
   doi={10.1006/jabr.2000.8529},
}
\bib{balsam2}{article}{ 
  author={Balsam, Benjamin},
  title={Turaev-Viro invariants as an extended TQFT II},
  date={2010-10},
}

\bib{balsam-kirillov}{article}{ 
   label={BalK2010},
  author={Balsam, Benjamin },
  author={Kirillov, Alexander, Jr},
  title={Turaev-Viro invariants as an extended TQFT},
  eprint={arXiv:1004.1533},
}
\bib{barrett}{article}{
   label={BW1996},
   author={Barrett, John W.},
   author={Westbury, Bruce W.},
   title={Invariants of piecewise-linear $3$-manifolds},
   journal={Trans. Amer. Math. Soc.},
   volume={348},
   date={1996},
   number={10},
   pages={3997--4022},
   issn={0002-9947},
   review={\MR{1357878 (97f:57017)}},
   doi={10.1090/S0002-9947-96-01660-1},
}
\bib{drinfeld}{article}{ 
  author={Drinfeld, Vladimir},
  author={Gelaki,Shlomo},
  author={Nikshych, Dmitri},
  author={Ostrik, Victor},
  title={On braided fusion categories I},
  eprint={arXiv:0906.0620},
}

		


\bib{ENO2005}{article}{
   label={ENO2005},
   author={Etingof, Pavel},
   author={Nikshych, Dmitri},
   author={Ostrik, Viktor},
   title={On fusion categories},
   journal={Ann. of Math. (2)},
   volume={162},
   date={2005},
   number={2},
   pages={581--642},
   issn={0003-486X},
   review={\MR{2183279 (2006m:16051)}},
   doi={10.4007/annals.2005.162.581},
}


\bib{ENO2009}{article}{
   label={ENO2009},
   author={Etingof, Pavel},
   author={Nikshych, Dmitry},
   author={Ostrik, Victor},
   title={Fusion categories and homotopy theory},
  eprint={arXiv:0909.3140},
}
\bib{FKLW}{article}{
   label={FKLW2003},
   author={Freedman, Michael H.},
   author={Kitaev, Alexei},
   author={Larsen, Michael J.},
   author={Wang, Zhenghan},
   title={Topological quantum computation},
   note={Mathematical challenges of the 21st century (Los Angeles, CA,
   2000)},
   journal={Bull. Amer. Math. Soc. (N.S.)},
   volume={40},
   date={2003},
   number={1},
   pages={31--38 (electronic)},
   issn={0273-0979},
   review={\MR{1943131 (2003m:57065)}},
   doi={10.1090/S0273-0979-02-00964-3},
}

\bib{FHK}{article}{
   label={FHK1994},
   author={Fukuma, M.},
   author={Hosono, S.},
   author={Kawai, H.},
   title={Lattice topological field theory in two dimensions},
   journal={Comm. Math. Phys.},
   volume={161},
   date={1994},
   number={1},
   pages={157--175},
   issn={0010-3616},
   review={\MR{1266073 (95b:81179)}},
}

\bib{FHLT}{article}{ 
  label={FHLT2009}, 
  author={Freed, Daniel},
  author={Hopkins, Michael},
  author={Lurie, Jacob},
  author={Teleman, Constantin},
  title={Topological quantum field theories from compact Lie groups},
  eprint={arXiv:0905.0731},
}

\bib{FNWW}{article}{
   label={FNWW2008},
   author={Freedman, M.},
   author={Nayak, C.},
   author={Walker, K.},
   author={Wang, Z.},
   title={On picture (2+1)-TQFTs},
   eprint={arXiv:0806.1926},
}
\bib{greenough}{article}{
   label={Gre2009},
   author={Justin Greenough},
   title={Monoidal 2-structure of Bimodule Categories},
  eprint={arXiv:0911.4979},
}

\bib{kadar}{article}{ 
  label={KMR2009},
  author={Kadar, Zoltan},
  author={Marzuoli, Annalisa},
  author={Rasetti, Mario},
  title={Microscopic description of 2d topological phases, 
          duality and 3d state sums},
  eprint={arXiv:0907.3724},
}


\bib{PLCW}{article}{ 
   label={Kir2010},
  author={Kirillov Jr, Alexander},
  title={On piecewise linear cell decompositions},
  eprint={arXiv:1009.4227},
  date={September 2010},
}

\bib{kitaev}{article}{ 
   label={Kit2003},
  author={Kitaev, A. Yu.},
  title={Fault-tolerant quantum computation by anyons},
  journal={Annals of Physics},
  volume={303}, 
  number={1},
  date={2003},
  pages={2-303},
}
\bib{kuperberg}{article}{ 
  label={KKR2010},
  author={Koenig, Robert},
  author={Kuperberg, Greg},
  author={Reichardt, Ben W.},
  title={Quantum computation with Turaev-Viro codes},
  eprint={arXiv:1002.2816},
}
\bib{levin-wen}{article}{ 
   label={LW2005},
  author={Levin, Michael},
  author={Wen, Xiao-Gang},
  title={String-net condensation: A physical mechanism for topological
phases},  journal={Phys. Rev. B},
  volume={71},
  number={4},
  date={2005},
  doi={10.1103/PhysRevB.71.045110},
}

\bib{lurie}{article}{ 
   label={Lur2009},
  author={Lurie, Jacob},
  title={On the classification of topological quantum field theories},
  eprint={http://www-math.mit.edu/~lurie/},
}
\bib{morrison-walker}{article}{ 
   label={MW2010},
  author={Morrison, Scott},
  author={Walker, Kevin},
  title={The blob complex},
  eprint={http://canyon23.net/math/},
}
\bib{muger1}{article}{   
   label={Mug2003a},
   author={M{\"u}ger, Michael},
   title={From subfactors to categories and topology. I. Frobenius algebras
   in and Morita equivalence of tensor categories},
   journal={J. Pure Appl. Algebra},
   volume={180},
   date={2003},
   number={1-2},
   pages={81--157},
   issn={0022-4049},
   review={\MR{1966524 (2004f:18013)}},
   doi={10.1016/S0022-4049(02)00247-5},
}


\bib{muger2}{article}{
   label={Mug2003b},
   author={M{\"u}ger, Michael},
   title={From subfactors to categories and topology. II. The quantum
double of tensor categories and subfactors},
   journal={J. Pure Appl. Algebra},
   volume={180},
   date={2003},
   number={1-2},
   pages={159--219},
   issn={0022-4049},
   review={\MR{1966525 (2004f:18014)}},
   doi={10.1016/S0022-4049(02)00248-7},
}




\bib{oeckl}{book}{
   author={Oeckl, Robert},
   title={Discrete gauge theory},
   note={From lattices to TQFT},
   publisher={Imperial College Press},
   place={London},
   date={2005},
   pages={xii+202},
   isbn={1-86094-579-1},
   review={\MR{2174961 (2006i:81142)}},
}

\bib{ostrik-module}{article}{
   author={Ostrik, Victor},
   title={Module categories, weak Hopf algebras and modular invariants},
   journal={Transform. Groups},
   volume={8},
   date={2003},
   number={2},
   pages={177--206},
   issn={1083-4362},
   review={\MR{1976459 (2004h:18006)}},
   doi={10.1007/s00031-003-0515-6},
}
	

		







\bib{stirling}{article}{
   label={Stir2010},
  author={Stirling, Spencer D.},
  title={Counterexamples in Levin-Wen string-net models, group categories,
         and Turaev unimodality},
  eprint={arXiv:1004.1737},
}
\bib{turaev93}{article}{
   label={Tur1993},
   author={Turaev, Vladimir},
   title={Quantum invariants of links and $3$-valent graphs in
   $3$-manifolds},
   journal={Inst. Hautes \'Etudes Sci. Publ. Math.},
   number={77},
   date={1993},
   pages={121--171},
   issn={0073-8301},
   review={\MR{1249172 (94j:57012)}},
}


\bib{turaev}{book}{
   label={Tur1994},
   author={Turaev, V. G.},
   title={Quantum invariants of knots and 3-manifolds},
   series={de Gruyter Studies in Mathematics},
   volume={18},
   publisher={Walter de Gruyter \& Co.},
   place={Berlin},
   date={1994},
   pages={x+588},
   isbn={3-11-013704-6},
   review={\MR{1292673 (95k:57014)}},
}

\bib{TV}{article}{
   label={TV1992},
   author={Turaev, V. G.},
   author={Viro, O. Ya.},
   title={State sum invariants of $3$-manifolds and quantum $6j$-symbols},
   journal={Topology},
   volume={31},
   date={1992},
   number={4},
   pages={865--902},
   issn={0040-9383},
   review={\MR{1191386 (94d:57044)}},
   doi={10.1016/0040-9383(92)90015-A},
}

\bib{TV2}{article}{
   label={TV2010},
  author={Turaev, Vladimir},
  author={Virelizier, Alexis},
  title={On two approaches to 3-dimensional TQFTs},
  eprint={arXiv:1006.3501},
}




\bib{walker}{article}{ 
   label={Wal2010},
  author={Walker, Kevin},
  title={TQFTs [early incomplete draft], version 1h},
  eprint={http://canyon23.net/math/},
}




\end{biblist}
\end{bibdiv}
\end{document}